%% file: SMOP_Accepted_240319.tex
\documentclass[onefignum,onetabnum]{siamart171218}
\newcommand{\norm}[1]{\left\lVert#1\right\rVert}
\usepackage{enumerate}
\usepackage{booktabs}
\usepackage{multirow}
\usepackage{multicol}
\usepackage{rotating}
 \usepackage{bigstrut}
 \usepackage{subcaption}
 \usepackage{tikz}
 \usetikzlibrary{shapes.geometric}

\input{ex_shared}

\definecolor{blue}{rgb}{0,0,0.9}
\definecolor{red}{rgb}{0.9,0,0}
\definecolor{green}{rgb}{0.0, 0.5, 0.0}

\ifpdf
\hypersetup{
  pdftitle={A highly efficient sieving based secant method for sparse optimization problems with least-squares constraints},
  pdfauthor={Authors}
}
\fi

\usepackage{amssymb}

\begin{document}

\maketitle

\begin{abstract}
  In this paper, we propose an efficient sieving based secant method to address the computational challenges of solving sparse optimization problems with least-squares constraints. A level-set method has been introduced in [X. Li, D.F. Sun, and K.-C. Toh, SIAM J. Optim., 28 (2018), pp. 1842--1866] that solves these problems by using the bisection method to find a root of a univariate nonsmooth equation $\varphi(\lambda) = \varrho$ for some $\varrho > 0$, where $\varphi(\cdot)$ is the value function computed by a solution of the corresponding regularized least-squares optimization problem.  When the objective function in the constrained problem is a polyhedral gauge function, we prove that (i) for any positive integer $k$, $\varphi(\cdot)$ is piecewise $C^k$ in an open interval containing the solution $\lambda^*$ to the equation $\varphi(\lambda) = \varrho$; (ii) the Clarke Jacobian of $\varphi(\cdot)$ is always positive. These results allow us to establish the essential ingredients of the fast convergence rates of the secant method. Moreover, an adaptive sieving technique is incorporated into the secant method to effectively reduce the dimension of the level-set subproblems for computing the value of $\varphi(\cdot)$. The high efficiency of the proposed algorithm is demonstrated by extensive numerical results.
\end{abstract}



\section{Introduction}
In this paper, we consider the following least-squares constrained optimization problem
\begin{equation}
\tag{\mbox{CP($\varrho$)}}
\label{eq: main-prob}
\min_{x \in \mathbb{R}^n} ~ \left\{p(x) \,|\, \|Ax - b\| \leq \varrho\right\},
\end{equation}
where $A \in \mathbb{R}^{m \times n}$ and $b \in \mathbb{R}^m$ are given data, $\varrho$ is a given parameter satisfying $0 < \varrho < \|b\|$, and $p: \mathbb{R}^n \to (-\infty, +\infty]$ is a proper closed convex function with $p(0) = 0$ that possesses the property of promoting sparsity. {Without loss of generality}, we assume that \cref{eq: main-prob} admits active solutions here.

Let $\lambda > 0$ be a given positive parameter.
Compared to the following regularized problem of the form
\begin{equation}
\tag{\mbox{${\rm P_{LS}}(\lambda)$}}
\label{eq: reg-LS-form}
\min_{x \in \mathbb{R}^n} ~ \left\{\frac{1}{2}\|Ax - b\|^2 + \lambda p(x)\right\},
\end{equation}
the constrained optimization problem \cref{eq: main-prob} is usually preferred in practical modeling since we can regard $\varrho$ as the noise level, which can be estimated in many applications. However, the optimization problem \cref{eq: main-prob} is {perceived to be} more challenging to solve in general due to the complicated geometry of the feasible set \cite{Aravkin18levelset}. Some algorithms such as the alternating direction method of multipliers (ADMM) \cite{glowinski1975approximation,gabay1976dual} are applicable to solve \cref{eq: main-prob}. {N}evertheless, to obtain an acceptable solution remains challenging {for these algorithms}.
In particular, when applying the ADMM for solving \cref{eq: main-prob}, it is computationally expensive to form the matrix $AA^T$ or to solve the linear systems involved in the subproblems. Recently, the dimension reduction techniques, such as the adaptive sieving \cite{AS-YLST23,AS-YST22}, have achieved some success in solving large-scale sparse optimization problems numerically by exploiting the solution sparsity. But it is still unclear how to apply dimension reduction techniques to \cref{eq: main-prob} due to the potential {infeasibility} issue for reduced problems.

A popular approach for solving \cref{eq: main-prob} and the more general convex constrained optimization problems is the level-set method \cite{Van08levelset,Van11levelset,Aravkin18levelset}, which has been widely used in many interesting applications \cite{Van08levelset,Van11levelset,Aravkin18levelset,li2018efficiently}. The idea of exchanging the role of the objective function and the constraints, which is the key for the level-set method, has a long history and can date back to Queen Dido's problem (see \cite[Page 548]{Richard01}). Readers can refer to \cite[Section 1.3]{Aravkin18levelset} and the references therein for a discussion of the history of the level-set method. In particular, the level-set method developed in \cite{Van08levelset,Van11levelset} solves the optimization problem \cref{eq: main-prob} by finding a root of the following univariate nonlinear equation
\begin{equation}
\tag{\mbox{$E_{\phi}$}}
\label{eq: root-finding-origin-levelset}
    \phi(\tau) = \varrho,
\end{equation}
where $\phi(\cdot)$ is the value function of the following level-set problem
\begin{equation}
\label{eq: origin-levelset-subprob}
\phi(\tau) := \min_{x \in \mathbb{R}^n} ~ \{\|Ax - b\| \,|\, p(x) \leq \tau\}, \quad \tau \geq 0.
\end{equation}
Therefore, by executing a root-finding procedure for \cref{eq: root-finding-origin-levelset} (e.g., the bisection method), one can obtain a solution to \cref{eq: main-prob} by solving a sequence of problems in the form of \cref{eq: origin-levelset-subprob} parameterized by $\tau$. In {implementations}, one needs {an efficient procedure} to compute the metric projection of given vector{s} onto the constraint set $\mathcal{F}_{p}(\tau) := \{x \in \mathbb{R}^n \,|\, p(x) \leq \tau, \, \tau > 0\}$. {However, such} an efficient computation procedure may not be available. One example can be found in \cite{li2018efficiently}, where $p(\cdot)$ is the fused Lasso regularizer \cite{Tibshiranifused05}. 
Moreover, it is still {not} clear {to us} how to {deal with the infeasibility issue when a dimension reduction technique is applied to} \cref{eq: origin-levelset-subprob}.

Recently, Li et al. \cite{li2018efficiently} proposed a level-set method for solving \cref{eq: main-prob} {via solving a sequence of \cref{eq: reg-LS-form}}. The dual of \cref{eq: reg-LS-form} can be written as
\begin{equation}
\tag{\mbox{${\rm D_{LS}}(\lambda)$}}
\label{eq: dual-LS-reg}
 \max_{y \in \mathbb{R}^m, u \in \mathbb{R}^n} \left\{-\frac{1}{2}\|y\|^2 + \langle b, y \rangle - \lambda p^*(u)\,|\, A^Ty - \lambda u = 0\right\},
\end{equation}
where $p^*(\cdot)$ is the Fenchel conjugate function of $p(\cdot)$, i.e., $p^*(z) = \sup_{x \in \mathbb{R}^n} ~ \{\langle z, x \rangle - p(x)\}, \, z \in \mathbb{R}^n$.
Let $\Omega(\lambda)$ be the solution set to \cref{eq: reg-LS-form}. Define the gauge $\Upsilon(\cdot~|~ C)$ of a nonempty convex set $C \subseteq \mathbb{R}^n$ as $\Upsilon(x ~|~ C) := \inf\{\nu \geq 0 ~|~ x \in \nu C\}, \, x \in \mathbb{R}^n$. Denote $\partial p(0)$ as the subdifferential of $p(\cdot)$ at the origin. In this paper, we assume
\begin{equation}
\label{eq: def-lambda-inf}
\lambda_{\infty} := \Upsilon(A^T b\,|\,\partial p(0)) \in (0, +\infty)
\end{equation}
and {that} for any $\lambda^\prime > 0$, there exists $(y(\lambda^\prime ), u(\lambda^\prime ), x(\lambda^\prime )) \in \mathbb{R}^m \times \mathbb{R}^n \times \mathbb{R}^n$ satisfying the following Karush-Kuhn-Tucker (KKT) system
\begin{equation}
\tag{KKT}
\label{eq: KKT-PD}
x \in \partial p^*(u), \quad y - b + Ax = 0, \quad A^T y  - \lambda^\prime u = 0,
\end{equation}
where $\partial p^*(\cdot)$ is the subdifferential of $p^*(\cdot)$. Consequently, the solution set $\Omega(\lambda)$ to \cref{eq: reg-LS-form} is  nonempty, and $b - Ax(\lambda)$ is invariant for all $x(\lambda) \in \Omega(\lambda)$ since the solution $(y(\lambda), u(\lambda))$ to \cref{eq: dual-LS-reg} is unique. Based on this fact, Li et al. \cite{li2018efficiently} proposed to solve \cref{eq: main-prob} by finding the root of the following equation:
\begin{equation}
\tag{\mbox{$E_{\varphi}$}}
\label{eq: root-finding-levelset-lst}
\varphi(\lambda) := \|Ax(\lambda) - b\| = \varrho,
\end{equation}
where $x(\lambda) \in \Omega(\lambda)$ is any solution to \cref{eq: reg-LS-form}. We assume that \cref{eq: root-finding-levelset-lst} has at least one solution $\lambda^* > 0$. We then know that any $x(\lambda^*) \in \Omega(\lambda^*)$ is a solution to \cref{eq: main-prob} \cite{li2018efficiently,Exactreg07}. There are {several} advantages to this approach. Firstly, it requires computing the proximal mapping of $p(\cdot)$, which is normally easier than computing the projection over the constraint set of \cref{eq: origin-levelset-subprob}. Secondly, efficient algorithms are available to solve the regularized least-squares problem \cref{eq: reg-LS-form} for a wide class of functions $p(\cdot)$ \cite{li2018highly,li2018efficiently,luo2019solving,Zhang2020,beck2009fast,glowinski1975approximation}.
More importantly, this approach is well-suited for applying dimension reduction techniques to solve \cref{eq: reg-LS-form} {as can be seen in subsequent sections}.

In this paper, we propose an efficient sieving based secant method for solving \cref{eq: main-prob} by finding the root of \cref{eq: root-finding-levelset-lst}. We call our algorithm \textbf{SMOP} as it is a root finding based \textbf{S}ecant \textbf{M}ethod for solving  the \textbf{O}ptimization \textbf{P}roblem \cref{eq: main-prob}.  We focus on the case where $p(\cdot)$ is a gauge function (see \cite[Section 15]{rockafellar1970convex}), i.e., $p(\cdot)$ is a nonnegative positively homogeneous convex function with $p(0) = 0$. We start by studying the properties of the value function $\varphi(\cdot)$ and the convergence rates of the secant method for solving \cref{eq: root-finding-levelset-lst}.
To address the computational challenges for solving \cref{eq: reg-LS-form} and computing the function value of $\varphi(\cdot)$, we incorporate an adaptive sieving (AS) technique \cite{AS-YLST23,AS-YST22} into the secant method to effectively reduce the dimension of \cref{eq: reg-LS-form}. The AS technique can exploit the sparsity of the solution of \cref{eq: reg-LS-form} so that {one} can obtain a solution to \cref{eq: reg-LS-form} by solving a sequence of reduced problems with much smaller dimensions. Extensive numerical results will be presented in this paper to demonstrate the superior performance of the proposed algorithm in solving \cref{eq: main-prob}.

The main contributions of this paper can be summarized in the following:
\begin{itemize}
    \item[1.] When $p(\cdot)$ is a gauge function, we prove that $\varphi(\cdot)$ is (strongly) semismooth for a wide class of instances of $p(\cdot)$ via connecting \ref{eq: dual-LS-reg} to a {metric} projection problem. More importantly, when $p(\cdot)$ is a polyhedral gauge function, we show that $\varphi(\cdot)$ is locally piecewise $C^k$ on $(0, \lambda_{\infty})$ for any integer $k \geq 1$; and for any $\bar{\lambda} \in (0, \lambda_{\infty})$, $v > 0$ for any $v \in \partial\varphi(\bar{\lambda})$.
    \item[2.] {Under the assumption that} $p(\cdot)$ is a polyhedral gauge function, we {show} that the secant method converges at least 3-step Q-quadratically for solving \cref{eq: root-finding-levelset-lst}, and if $\partial_{\rm B}\varphi(\lambda^*)$ is a singleton, the secant method converges superlinearly with Q-order at least $(1+\sqrt{5})/2$. Furthermore, for a general strongly semismooth function $\varphi(\cdot)$, if $\partial \varphi(\lambda^*)$ is a singleton and nondegenerate, the secant method converges superlinearly with R-order of at least $(1+\sqrt{5})/2$.
    \item[3.] We propose an efficient sieving based secant method to address the computational challenges for solving \cref{eq: main-prob}. The algorithm incorporates a fast convergent secant method for root-finding of \cref{eq: root-finding-levelset-lst}, along with an AS technique for effectively reducing the dimension of subproblems in the form of \cref{eq: reg-LS-form}. The efficiency of the proposed algorithm for solving \cref{eq: main-prob} will be demonstrated by extensive numerical results.
\end{itemize}

The rest of the paper is organized as follows. We will introduce some necessary preliminary results in \cref{sec: preliminary}. We discuss the properties of the value function $\varphi(\cdot)$ in \cref{sec: property-value-function,sec: HS-Jacobian}. A secant method for solving \cref{eq: main-prob} will be introduced and analyzed in \cref{sec: secant-levelset}. We will introduce the AS technique in \cref{sec: AS-levelset} followed by {presenting} extensive numerical results in \cref{sec: numerical-experiment}. {We} conclude the paper in \cref{sec: conclusion}.

\noindent \textbf{Notation}. Let $n \geq 1$ be any given integer. Denote the nonnegative orthant and the positive orthant of $\mathbb{R}^n$ as $\mathbb{R}^n_+$ and $\mathbb{R}^n_{++}$, respectively. We denote $[n] := \{1, 2, \dots, n\}$. We denote the subvector generated by $x \in \mathbb{R}^n$ indexed by $K \subseteq [n]$ as $x_K$ and submatrix generated by the columns (rows) of $A \in \mathbb{R}^{m \times n}$ indexed by $K \subseteq [n]$ ($K \subseteq [m]$) as $A_{:K}$ ($A_{K:}$). For any $x \in \mathbb{R}^n$ and any integer $q \geq 1$, the $\ell_q$ norm of $x$ is defined as $\|x\|_q := \sqrt[q]{|x_1|^q + \cdots + |x_n|^q}$. We denote $\|\cdot\| = \|\cdot\|_2$. Let $U \subseteq \mathbb{R}^n$ be an open set. We say that a function $f: U \to \mathbb{R}$ is $C^k$ for some integer $k \geq 1$ if $f{(\cdot)}$ is $k$-times continuously differentiable on $U$. Let  $p: \mathbb{R}^n \to (-\infty, +\infty]$ be a proper closed convex function. The proximal mapping of {$p(\cdot)$} is defined by
\[
{\rm Prox}_p(z) := \arg\min_{x \in \mathbb{R}^n} \left\{ p(x) + \frac{1}{2}\|x - z\|^2 \right\}, \quad z \in \mathbb{R}^n.
\]
The polar of a gauge function $p(\cdot)$ is defined by \[p^{\circ}(y) := \inf \{\nu \geq 0 \,|\, \langle y, x \rangle \leq \nu p(x), ~ \forall \, x \in \mathbb{R}^n\}, \quad y \in \mathbb{R}^n.\] Let $C \subseteq \mathbb{R}^n$ be a {nonempty} closed convex set. We use $\Pi_C(\cdot)$ to denote the metric projector over $C$. The indicator function of $C$ is defined as
\[
\delta(x \,|\, C) := \left\{
\begin{array}{ll}
 0    & \mbox{if $x \in C$},\\
 +\infty    & \mbox{otherwise}.
\end{array}
\right.
\]

\section{Preliminaries}
\label{sec: preliminary}
Let $\mathcal{X}$ and $\mathcal{Y}$ be two finite dimensional real vector spaces each equipped with a scalar product $\langle \cdot, \cdot \rangle$ and its induced norm $\|\cdot\|$. Denote the set of all linear operators from $\mathcal{X}$ to $\mathcal{Y}$ by $\mathcal{L}(\mathcal{X}, \mathcal{Y})$. We first review some preliminary results related to the generalized Jacobians and the semismoothness. Let $\mathcal{O} \subseteq \mathcal{X}$ be an open set and $F: \mathcal{O} \to \mathcal{Y}$ be a locally Lipschitz continuous function on $\mathcal{O}$. According to Rademacher's theorem, $F(\cdot)$ is differentiable (in the sense of Fr$\rm \acute{e}$chet) almost everywhere on $\mathcal{O}$. Denote by $D_F$ the set of all points on $\mathcal{O}$ where $F(\cdot)$ is differentiable.
Denote $F'(x)$ as the Jacobian of $F(\cdot)$ at $x \in D_F$. Define the B-subdifferential of $F(\cdot)$ at $x \in \mathbb{R}^n$ as
\begin{equation}
	\label{def: B-subdifferential}
	\partial_{\rm B} F(x) := \left\{\lim_{x^k \to x, x^k \in D_F} F'(x^k)\right\}.
\end{equation}
The Clarke generalized Jacobian of $F(\cdot)$ at $x \in \mathbb{R}^n$ is then defined as follows \cite{clarke1983optimization}:
\begin{equation}
	\label{def: Clarke-Jacobian}
	\partial F(x) := {\rm conv}(\partial_{\rm B} F(x)).
\end{equation}
Note that both $\partial_{\rm B} F(\cdot)$ and $\partial F(\cdot)$ are compact valued and upper-semicontinuous multi-valued functions. For finitely valued convex functions, the Clarke generalized Jacobian coincides with the subdifferential in the sense of convex analysis \cite[Proposition 2.2.7]{clarke1983optimization}.
Now, we introduce the concept of G-semismoothness (with respect to a multifunction).
\begin{definition}\label{def: semismooth}
	Let $\mathcal{O}\subseteq \mathcal{X}$ be an open set, $F:\ \mathcal{O}\rightarrow \mathcal{Y}$ be a locally Lipschitz continuous function, and $T_F:\ \mathcal{O}\rightrightarrows \mathcal{L}(\mathcal{X}, \mathcal{Y})$ be a nonempty and compact valued, upper-semicontinuous set-valued mapping. $F(\cdot)$ is said to be G-semismooth at $x\in\mathcal{O}$ with respect to the multifunction $T_F$,  if for any $V\in T_F(x+\Delta x)$ with $\Delta x\rightarrow0$,
	\begin{equation}
        \label{def: G-semismooth}
		F(x+\Delta x)-F(x)-V\Delta x=o(\|\Delta x\|).
	\end{equation}
	Let $\gamma > 0$ be a constant. $F(\cdot)$ is said to be $\gamma$-order (strongly, if $\gamma = 1$) G-semismooth at $x\in\mathcal{O}$ with respect to $T_F$ if for any $V\in T_F(x+\Delta x)$ with $\Delta x\rightarrow 0$,
	\begin{equation}
		F(x+\Delta x)-F(x)-V\Delta x=O(\|\Delta x\|^{1+\gamma}).
	\end{equation}
 $F(\cdot)$ is said to be G-semismooth (respectively, $\gamma$-order G-semismooth, strongly G-semismooth) on $\mathcal{O}$ with respect to the multifunction $T_F$ if it is G-semismooth (respectively, $\gamma$-order G-semismooth, strongly G-semismooth) everywhere in $\mathcal{O}$ with respect to the multifunction $T_F$. All the above definitions of G-semimsoothness will be replaced by semismoothness if $F(\cdot)$ happens to be directionally differentiable at the concerned point $x \in \mathcal{O}$.
\end{definition}

Let $f: \mathbb{R} \to \mathbb{R}$ be a {real valued functional}. Denote for $x \neq y$ that
\begin{equation}
    \delta_f(x, y) := \left(f(x) - f(y)\right) / (x - y).
\end{equation}
The following lemma is useful {for analyzing the convergence of the secant method}. Part (i) of this lemma is from \cite[Lemma 2.2]{Potra98}, and Part (ii) can be proved by following a similar procedure as in the proof of \cite[Lemma 2.3]{Potra98}.

\begin{lemma}
\label{lemma: lateral-derivative-approx}
Assume that $f: \mathbb{R} \to \mathbb{R}$ is semismooth at $\bar{x} \in \mathbb{R}$. Denote the lateral derivatives of $f$ at $\bar x$ by
    \begin{equation}
    \label{def: lateral-derivative}
        \bar{d}^{-} := - f'(\bar{x}; -1)\quad  \mbox{ and } \quad \bar{d}^{+} := f'(\bar{x}; 1).
    \end{equation}
\begin{romannum}
    \item Then the lateral derivatives $\bar{d}^{-} $ and $\bar{d}^{+}$ exist and
    \[
    \partial_{\rm B}f(\bar{x}) = \{\bar{d}^{-}, \bar{d}^{+}\};
    \]
    \item It holds that
    \begin{eqnarray}
    \bar{d}^{-} - \delta_f(u, v) = o(1) \quad \mbox{for all $u \uparrow \bar{x}$, $v \uparrow \bar{x}$};\\\label{eq: delta-o-}
    \bar{d}^{+} - \delta_f(u, v) = o(1) \quad \mbox{for all $u \downarrow \bar{x}$, $v \downarrow \bar{x}$};
    \end{eqnarray}
    moreover, if $f(\cdot)$ is $\gamma$-order semismooth at $\bar{x}$ for some $\gamma > 0$, then
    \begin{eqnarray}
    \bar{d}^{-} - \delta_f(u, v) = O(|u - \bar{x}|^{\gamma} + |v - \bar{x}|^{\gamma}) \quad \mbox{for all $u \uparrow \bar{x}$, $v \uparrow \bar{x}$};\\
    \bar{d}^{+} - \delta_f(u, v) = O(|u - \bar{x}|^{\gamma} + |v - \bar{x}|^{\gamma}) \quad \mbox{for all $u \downarrow \bar{x}$, $v \downarrow \bar{x}$}.
    \end{eqnarray}
\end{romannum}
\end{lemma}

\section{Properties of the value function $\varphi(\cdot)$}
\label{sec: property-value-function}
In this section, we first discuss some useful properties of the function $\varphi(\cdot)$. Since $p(\cdot)$ is assumed to be a nonnegative convex function with $p(0) = 0$, we know that $0 \in \partial p(0)$.

\begin{proposition}
\label{prop:value-fun-lst18}
Assume that $0 < \lambda_{\infty} < +\infty$. It holds that
\begin{romannum}
    \item for all $\lambda \geq \lambda_{\infty}$, $y(\lambda) = b$ and $0 \in \Omega(\lambda)$;
    \item the value function $\varphi(\cdot)$ is nondecreasing on $(0, +\infty)$ and for any $0 < \lambda_1 < \lambda_2 < +\infty$, $\varphi(\lambda_1) = \varphi(\lambda_2)$ implies $p(x(\lambda_1)) = p(x(\lambda_2))$, where for any $\lambda > 0$, $x(\lambda)$ is an optimal solution to \cref{eq: reg-LS-form}.
\end{romannum}
\end{proposition}

\begin{proof}
(i) Since $0 \in \partial p(0)$, for all $\lambda > \lambda_{\infty}$, it holds that
\[A^T b/\lambda \in \partial p(0),\]
which implies that $0 \in \Omega(\lambda)$. Since $\lambda_{\infty} > 0$ and $\partial p(0)$ is closed, we know
\[
A^Tb/\lambda_{\infty} \in \partial p(0),
\]
which implies that $0 \in \Omega(\lambda_{\infty})$. Therefore, for all $\lambda \geq \lambda_{\infty}$, $0 \in \Omega(\lambda)$ and $y(\lambda) = b$ .

(ii) Let $0 < \lambda_1 < \lambda_2 < \infty$ be arbitrarily chosen. Let $x(\lambda_1) \in \Omega(\lambda_1)$ and $x(\lambda_2) \in \Omega(\lambda_2)$. Then, we have
\begin{eqnarray}\label{ineq: part1}
\frac{1}{2}\|Ax(\lambda_1) - b\|^2 + \lambda_1 p(x(\lambda_1)) &\leq& \frac{1}{2}\|Ax(\lambda_2) - b\|^2 + \lambda_1 p(x(\lambda_2)),\\ \label{ineq: part2}
\frac{1}{2}\|Ax(\lambda_2) - b\|^2 + \lambda_2 p(x(\lambda_2)) &\leq& \frac{1}{2}\|Ax(\lambda_1) - b\|^2 + \lambda_2 p(x(\lambda_1)),
\end{eqnarray}
which implies that
\begin{equation}
(\lambda_1 - \lambda_2)(p(x(\lambda_1)) - p(x(\lambda_2))) \leq 0.
\end{equation}
Since $\lambda_1 - \lambda_2 <0$, we know that $p(x(\lambda_1)) \geq p(x(\lambda_2))$. It follows from \eqref{ineq: part1} that
\[
\frac{1}{2}\|Ax(\lambda_1) - b\|^2 \leq \frac{1}{2}\|Ax(\lambda_2) - b\|^2 + \lambda_1(p(x(\lambda_2)) - p(x(\lambda_1))) \leq \frac{1}{2}\|Ax(\lambda_2) - b\|^2,
\]
which implies that $p(x(\lambda_1)) = p(x(\lambda_2))$ if $\varphi(\lambda_1) = \varphi(\lambda_2)$. This completes the proof of the proposition.
\end{proof}

Due to \cref{prop:value-fun-lst18}, we can apply the bisection method to solve \cref{eq: root-finding-levelset-lst} and for any $\epsilon > 0$ we can obtain a solution $\lambda_{\varepsilon}$ satisfying $|\lambda_{\varepsilon} - \lambda^*| \leq \varepsilon$ in $O(\log(1/\varepsilon))$ iterations, where $\lambda^*$ is a solution to \cref{eq: root-finding-levelset-lst}. In this paper, we will design a more efficient secant method for solving \cref{eq: root-finding-levelset-lst}. To achieve this goal, we first study the (strong) semismoothness property of $\varphi(\cdot)$.

In this paper, we focus on the case where $p(\cdot)$ is a gauge function.
In most of the applications, $p(\cdot)$ is a norm {function}, which is {automatically} a gauge function.
We will leave the study of the (strong) semismoothness of $\varphi(\cdot)$ for a general $p(\cdot)$ as future work. Because of its repeated occurrence in this section, we present the following assumption:

\textbf{Assumption A}:
    Assume that $p(\cdot)$ is a gauge function.

Under Assumption A, $p^*(\cdot) = \delta(\cdot\,|\,\partial p(0))$ and the optimization problem \cref{eq: dual-LS-reg} is equivalent to
\begin{equation}
\label{eq: dual-LS-reg-gauge}
 \max_{y \in \mathbb{R}^m} \left\{-\frac{1}{2}\|y\|^2 + \langle b, y \rangle \ |\ \lambda^{-1} y \in Q\right\},
\end{equation}
where
\begin{equation}
\label{def: set-Q}
Q := \{z \in \mathbb{R}^m \ |\ A^Tz \in \partial p(0)\}.
\end{equation}
Then by performing a variable substitution, we have the following useful observation about the solution mapping to \cref{eq: dual-LS-reg-gauge}.
\begin{proposition}
\label{prop: dual-projection}
Under Assumption A, for any $0 < \lambda < +\infty$, the unique solution to \cref{eq: dual-LS-reg-gauge} can be written as
\begin{equation}
\label{solution-projection-y}
y(\lambda) = \lambda \, \Pi_{Q}(\lambda^{-1}b) = \Pi_{\lambda Q}(b).
\end{equation}
\end{proposition}

The following proposition is useful in understanding the semismoothness of $y(\cdot)$ and $\varphi(\cdot)$ even if $p$ is non-polyhedral. The definition of a tame set and a globally subanalytic set can be found in \cite[Definition 2]{Bolte09} and \cite[Example 2(a)]{Bolte09}, respectively.
Part (ii) of the proposition is generalized from \cite[Proposition 1 (iv)]{li2018efficiently} ($p(\cdot)$ is assumed to be a polyhedral gauge function in \cite{li2018efficiently}) and we provide a more explicit proof that does not rely on the piecewise linearity of the solution mapping $y(\cdot)$ in \cref{solution-projection-y}.
\begin{proposition}
\label{prop: strict-monotonicity-tame}
Under Assumption A, it holds that
\begin{romannum}
    \item the functions $y(\cdot)$ and $\varphi(\cdot)$ are locally Lipschitz continuous on $(0, +\infty)$;
    \item if $0 < \lambda_{\infty} < +\infty$, $\varphi(\cdot)$ is strictly increasing on $(0, \lambda_{\infty}]$;
    \item if the set $Q$ is tame, $\varphi(\cdot)$ is semismooth on $(0, +\infty)$;
    \item if $Q$ is globally subanalytic, $\varphi(\cdot)$ is $\gamma$-order semismooth on $(0, +\infty)$ for some $\gamma > 0$.
\end{romannum}
\end{proposition}

\begin{proof}
For convenience, we denote $\tilde{y}(\lambda) := \Pi_{Q}(\lambda^{-1}b)$ for any $\lambda > 0$.

(i) Since $\Pi_{Q}(\cdot)$ is Lipschitz continuous with modulus $1$, both $\tilde{y}(\cdot)$ and $y(\cdot)$ are locally Lipschitz continuous on $(0, +\infty)$. Therefore, $\varphi(\cdot) = \|y(\cdot)\|$ is locally Lipschitz continuous on $(0, +\infty)$.

(ii) It follows from \cref{prop:value-fun-lst18} that $\varphi(\cdot)$ is nondecreasing. We will now prove that $\varphi(\cdot)$ is strictly increasing on $(0, \lambda_{\infty}]$. We prove it by contradiction. Assume that there exist $0 < \lambda_1 < \lambda_2 \leq \lambda_\infty$ such that $\varphi(\lambda_1) = \varphi(\lambda_2)$. Let $x(\lambda_1) \in \Omega(\lambda_1)$ and $x(\lambda_2) \in \Omega(\lambda_2)$ be arbitrarily chosen. From \cref{prop:value-fun-lst18} (ii), we know that
\[
	p(x(\lambda_1)) = p(x(\lambda_2)),
\]
which implies that $x(\lambda_1) \in \Omega(\lambda_2)$ and $x(\lambda_2) \in \Omega(\lambda_1)$. Therefore, we get
\[
y(\lambda_1) = b - Ax(\lambda_1) = b - Ax(\lambda_2) = y(\lambda_2).
\]
Thus, by using the facts $0 \in \lambda_2 Q$, $y(\lambda_2) = \Pi_{\lambda_2Q}(b)$, and $\lambda_1^{-1}\lambda_2y(\lambda_2) = \lambda_1^{-1}\lambda_2y(\lambda_1) \in \lambda_2Q$, we obtain from the properties of the metric projector $\Pi_{\lambda_2 Q}(\cdot)$ that 
\[
\langle b - y(\lambda_2), 0 - y(\lambda_2) \rangle \leq 0, \quad \langle b - y(\lambda_2), \lambda_1^{-1}\lambda_2y(\lambda_2) - y(\lambda_2) \rangle \leq 0.
\]
Therefore,
\[
\langle b - y(\lambda_1), y(\lambda_1) \rangle = \langle b - y(\lambda_2), y(\lambda_2) \rangle = 0.
\]
Since $\lambda_1 < \lambda_{\infty}$ and $y(\lambda_1) = \Pi_{\lambda_1Q}(b)$, we know that $y(\lambda_1) \neq b$. Hence,
\[
\langle b - y(\lambda_1), \lambda_{\infty}^{-1}\lambda_1b - y(\lambda_1)\rangle = \lambda_{\infty}^{-1}\lambda_1\|b - y(\lambda_1)\|^2 > 0.
\]
However, $\lambda_{\infty}^{-1}\lambda_1b \in \lambda_1Q$ and $y(\lambda_1) = \Pi_{\lambda_1Q}(b)$ imply that
\[
\langle b - y(\lambda_1), \lambda_{\infty}^{-1}\lambda_1b - y(\lambda_1) \rangle \leq 0,
\]
which is a contradiction. This contradiction shows that $\varphi(\cdot)$ is strictly increasing on $(0, \lambda_{\infty}]$.

(iii) Since $\tilde{y}(\cdot)$ is locally Lipschitz continuous on $(0, \lambda_{\infty})$, it follows from \cite{Bolte09} that $\tilde{y}(\cdot)$ is semismooth on $(0, +\infty)$ if $Q$ is a tame set. Since $\|\cdot\|$ is strongly semismooth, we know that $g_1(\cdot) = \|\tilde{y}(\cdot)\|$ is semismooth on $(0, \lambda_{\infty})$. Therefore, $\varphi(\cdot)$ is semismooth on $(0, \lambda_{\infty})$ \cite[Proposition 7.4.4, Proposition 7.4.8]{Facchinei03}.

(iv) The $\gamma$-order semismoothness of $\varphi(\cdot)$ can be proved similarly as for (iii).
\end{proof}

The above proposition can be used to prove the semismoothness of $\varphi(\cdot)$ for a wide class of {functions} $p(\cdot)$. For example, the next corollary shows the semismoothness of $\varphi(\cdot)$ when $p(\cdot)$ is the nuclear norm function defined on $\mathbb{R}^{d \times n}$, using the fact that $\partial p(0)$ is linear matrix inequality representable \cite{Fazel10nuclear}.

\begin{corollary}
Denote the adjoint of the linear operator $\mathcal{A}: \mathbb{R}^{d \times n} \to \mathbb{R}^m$ as $\mathcal{A}^*$. Let $p(\cdot) = \|\cdot\|_{*}$ be the nuclear norm function defined on $\mathbb{R}^{d \times n}$. Then $Q = \{z \in \mathbb{R}^m \,|\, \mathcal{A}^*z \in \partial p(0)\}$ is a tame set and $\Pi_{Q}(\cdot)$ is semismooth.
\end{corollary}

Next, we will show that $\varphi(\cdot)$ can be strongly semismooth for a class of important instances of $p(\cdot)$.

\begin{proposition}
\label{prop: strong-semismoothness-varphi}
Define $\Phi(x) := \frac{1}{2}\|Ax - b\|^2, \, x \in \mathbb{R}^n$ and
\[
H(x, \lambda) := x - {\rm Prox}_{p}(x - \lambda^{-1}\nabla\Phi(x)), \quad (x, \lambda) \in \mathbb{R}^n \times \mathbb{R}_{++}.
\]
For any $(x, \lambda) \in \mathbb{R}^n \times \mathbb{R}_{++}$, denote $\partial_x H(x, \lambda)$ as the Canonical projection of $\partial H(x, \lambda)$ onto $\mathbb{R}^n$. Under Assumption A, it holds that
\begin{romannum}
    \item if $\Pi_{\partial p(0)}(\cdot)$ is strongly semismooth and $\partial_x H(\bar{x}, \bar{\lambda})$ is nondegenerate at some $(\bar{x}, \bar{\lambda})$ satisfying $H(\bar{x}, \bar{\lambda}) = 0$, then $y(\cdot)$ and $\varphi(\cdot)$ are strongly semismooth at $\bar{\lambda}$;
     \item if $p(\cdot)$ is further assumed to be polyhedral, the function $y(\cdot)$ is piecewise affine and $\varphi(\cdot)$ is strongly semismooth on $\mathbb{R}_{++}$.
\end{romannum}
\end{proposition}

\begin{proof}
(i) It follows from the Moreau identity \cite[Theorem 31.5]{rockafellar1970convex} that for any $(x, \lambda) \in \mathbb{R}^n \times \mathbb{R}_{++}$,
\[
\begin{array}{lcl}
H(x, \lambda) &=& x - ((x - \lambda^{-1}\nabla\Phi(x)) - {\rm Prox}_{p^*}(x - \lambda^{-1}\nabla\Phi(x)))\\
&=& \lambda^{-1}\nabla\Phi(x) + \Pi_{\partial p(0)}(x - \lambda^{-1}\nabla\Phi(x)).
\end{array}
\]
The rest of the proof can be obtained from the fact that $\nabla\Phi(\cdot)$ is linear and the Implicit Function Theorem for semismooth functions \cite{Sun01implicit,Meng05}.

(ii) When $p(\cdot)$ is a polyhedral gauge function, we know that the set $Q$ defined in \cref{def: set-Q} is a convex polyhedral set \cite[Theorem 19.3]{rockafellar1970convex} and the projector $\Pi_{Q}(\cdot)$ is piecewise affine \cite[Proposition 4.1.4]{Facchinei03}. Therefore, $y(\cdot)$ is a piecewise affine function on $(0, +\infty)$. Then both $y(\cdot)$ and $\varphi(\cdot)$ are strongly semismooth on $(0, +\infty)$ \cite[Proposition 7.4.7, Proposition 7.4.4, Proposition 7.4.8]{Facchinei03}.
\end{proof}

\begin{remark}
We make some remarks on the assumptions {in} Part (i) of \cref{prop: strong-semismoothness-varphi}. On one hand, the strong semismoothness of the projector $\Pi_{K}(\cdot)$ has been proved for some important non-polyhedral closed convex sets $K$. In particular, $\Pi_K(\cdot)$ is strongly semismooth if $K$ is the positive semidefinite cone \cite{Sunsemismooth02}, the second-order cone \cite{Chen03SOC}, or the $\ell_2$ norm ball \cite[Lemma 2.1]{Zhang2020}. On the other hand, the assumption of the nondegeneracy of $\partial_x H(\cdot, \cdot)$ at the concerned point is closely related to the important concept of strong regularity of the KKT system of \cref{eq: reg-LS-form}. One can refer to the Monograph \cite{Bonnans00} and the references therein for a general discussion, and to \cite{Sun06strong,ChanSun08} for the semidefinite programming problems.
\end{remark}

\section{The HS-Jacobian of $\varphi(\cdot)$ for polyhedral gauge functions $p(\cdot)$}
\label{sec: HS-Jacobian}
In this section, we assume by default that $p(\cdot)$ is a polyhedral gauge function and $0< \lambda_{\infty} < +\infty$. {Then the} set $\partial p(0)$ is polyhedral \cite[Theorem 19.2]{rockafellar1970convex}, which can be assume{d} without loss of generality, to take the form of
\begin{equation}
\label{eq: representation-p-polar}
\partial p(0) := \left\{
u \in \mathbb{R}^n \,|\, Bu \leq d
\right\}
\end{equation}
for some $B \in \mathbb{R}^{q \times n}$ and $d \in \mathbb{R}^q$.

 Inspired by the generalized Jacobian for the projector over a polyhedral set derived by Han and Sun \cite{han1997newton}, which we call the HS-Jacobian, we will derive the HS-Jacobian of the function $\varphi(\cdot)$. As an important implication, we will prove that the Clarke Jacobian of $\varphi(\cdot)$ {at any $\lambda \in (0, \lambda_{\infty})$} is positive. Note that the open interval $(0, \lambda_{\infty})$ contains the solution $\lambda^*$ to \cref{eq: root-finding-levelset-lst}.

Let $\lambda \in (0, \lambda_{\infty})$ be arbitrarily chosen. Let $(y(\lambda), u(\lambda))$ be the unique solution to \cref{eq: dual-LS-reg} with the parameter $\lambda$. 
Here, we denote $(y, u) = (y(\lambda), u(\lambda))$ to simplify our notation and hide the dependency on $\lambda$.
Then there exists $x \in \Omega(\lambda)$ such that $(y, u, x)$ satisfies the following KKT system:
\begin{equation}
\label{KKT-general-org}
u = \Pi_{\partial p(0)}(u + x), \quad
y - b + Ax = 0, \quad
A^Ty - \lambda u = 0.
\end{equation}
Therefore, $u$ is the unique solution to the following optimization problem
\begin{equation}
\label{prob: project-D}
\min_{z \in \mathbb{R}^n} ~ \left\{\frac{1}{2}\|z - (u+x)\|^2 \,|\, Bz \leq d\right\}
\end{equation}
and there exists $\xi\in \mathbb{R}^q$ such that $(u, \xi)$ satisfies the following KKT system for \cref{prob: project-D}:
\begin{equation}
\label{eq: KKT-Projection-to-D}
B^T \xi - x = 0, \quad Bu - d \leq 0, \quad \xi \geq 0, \quad \xi^T(Bu - d) = 0.
\end{equation}
As a result, there exists $(x, \xi) \in \mathbb{R}^n \times \mathbb{R}^q$ such that $(y, u, x, \xi)$ satisfies the following augmented KKT system
\begin{equation}
\label{KKT-general-augmented}
\left\{
\begin{array}{l}
B^T \xi - x = 0, \quad Bu - d \leq 0, \quad \xi \geq 0, \quad \xi^T(Bu - d) = 0,\\[5pt]
y - b + Ax = 0, \quad A^Ty - \lambda u = 0.
\end{array}
\right.
\end{equation}
Let $M(\lambda)$ be the set of Lagrange multipliers associated with $(y, u)$ defined as
\begin{equation*}
M(\lambda) := \left\{(x, \xi) \in \mathbb{R}^n \times \mathbb{R}^q \,|\, \mbox{$(y, u, x, \xi)$ satisfies \cref{KKT-general-augmented}}
\right\}.
\end{equation*}
Since $x = B^T\xi$, we obtain the following system by eliminating the variable $x$ in \cref{KKT-general-augmented}:
\begin{equation}
    \label{KKT-general-augmented-reduced}
\left\{
\begin{array}{l}
Bu - d \leq 0, \quad \xi \geq 0, \quad \xi^T(Bu - d) = 0,\\[5pt]
y - b + \widehat{A} \xi = 0, \quad A^Ty - \lambda u = 0,
\end{array}
\right.
\end{equation}
where $\widehat{A} = AB^T \in \mathbb{R}^{m \times q}$. Denote
\begin{equation}
\widehat{M}(\lambda) := \left\{\xi \in \mathbb{R}^q \,|\, (y, u, \xi) \mbox{ satisfies } \cref{KKT-general-augmented-reduced}\right\}.
\end{equation}
Then, the set $M(\lambda)$ is equivalent to
\begin{equation}
\label{eq: M-gamma-augmented}
M(\lambda) = \left\{(x, \xi) \in \mathbb{R}^n \times \mathbb{R}^q \,|\, x = B^T\xi, ~ \xi \in \widehat{M}(\lambda)
\right\}.
\end{equation}
Denote the active set of $u$ as
\begin{equation}
\label{eq: I_u_lasso}
I(u) := \{i \in [q] \,|\, B_{i:}u - d_i = 0\}.
\end{equation}
For any $\lambda \in (0, \lambda_{\infty})$, we define
\begin{equation}
\label{B-set-lasso}
\mathcal{B}(\lambda) :=  \left\{
K \subseteq [q] \,|\, \mbox{$\exists ~ \xi \in \widehat{M}(\lambda)$ s.t. ${\rm supp}(\xi) \subseteq K \subseteq I(u)$} \mbox{ and ${\rm rank}(\widehat{A}_{:K}) = |K|$}
\right\}.
\end{equation}
Since the polyhedral set $\widehat{M}(\lambda)$ does not contain a line, this implies that $\widehat{M}(\lambda)$ has at least one extreme point $\bar{\xi}$ \cite[Corollary 18.5.3]{rockafellar1970convex}. Note that $0 < \lambda < \lambda_{\infty}$ and $x \neq 0$, which implies that $\bar{\xi} \neq 0$ and $\mathcal{B}(\lambda)$ is nonempty.

Define the HS-Jacobian of $y(\cdot)$ as
\begin{equation}
\label{eq: Jacobian-y-lasso}
    \mathcal{H}(\lambda) := \left\{
    h^K \in \mathbb{R}^m \,|\, h^K = \widehat{A}_{:K}(\widehat{A}_{:K}^T\widehat{A}_{:K})^{-1}d_K, ~ K \in \mathcal{B}(\lambda)
    \right\},
    \quad \lambda \in (0, \lambda_{\infty}),
\end{equation}
where $d_K$ is the subvector of $d$ indexed by $K$. For notational convenience, for any $\lambda \in (0, \lambda_{\infty})$ and $K \in \mathcal{B}(\lambda)$, denote
\begin{equation}
 P^K = I - \widehat{A}_{:K}(\widehat{A}_{:K}^T\widehat{A}_{:K})^{-1}\widehat{A}_{:K}^T.
\end{equation}
Define
\begin{equation}
	\label{eq: Jacobian-varphi}
	\mathcal{V}(\lambda) := \left\{
	t  \in \mathbb{R} \,|\, t = \lambda\|h\|^2/\varphi(\lambda), ~ h \in \mathcal{H}(\lambda)
	\right\}, \quad  \lambda \in \mathcal{D},
\end{equation}
where $\mathcal{D} = \{\lambda \in (0, \lambda_{\infty}) \,|\, \varphi(\lambda) > 0 \}$. The following lemma is proved by following the same line as in \cite[Lemma 2.1]{han1997newton}.
\begin{lemma}
\label{lemma: upper-semicontinuouity}
  Let $\bar{\lambda} \in (0, \lambda_{\infty})$ be arbitrarily chosen. It holds that
\begin{equation}
\label{eq: expression-y}
    y(\bar{\lambda}) = P^K b + \bar{\lambda} h^K, \quad \forall \, h^K \in \mathcal{H}(\bar{\lambda}).
\end{equation}
Moreover, there exists a positive scalar $\varsigma$ such that $\mathcal{N}(\bar{\lambda}) := (\bar{\lambda} - \varsigma, \bar{\lambda} + \varsigma) \subseteq (0, \lambda_{\infty})$ and for all $\lambda \in \mathcal{N}(\bar{\lambda})$,
\begin{romannum}
\item $\mathcal{B}(\lambda) \subseteq \mathcal{B}(\bar{\lambda}) \quad \mbox{and} \quad \mathcal{H}(\lambda) \subseteq \mathcal{H}(\bar{\lambda})$;
\item $y(\lambda) = y(\bar{\lambda}) + (\lambda - \bar{\lambda})h, \quad \forall \, h \in \mathcal{H}(\lambda)$.
\end{romannum}
\end{lemma}

\begin{proof}
Choose a sufficiently small $\varsigma > 0$ such that $\mathcal{N}(\bar{\lambda}) := (\bar{\lambda} - \varsigma, \bar{\lambda} + \varsigma) \subseteq (0, \lambda_{\infty})$ and let $\lambda \in \mathcal{N}(\bar{\lambda}) \backslash \bar{\lambda}$ be arbitrarily chosen. Denote $(\bar{y}, \bar{u}) = (y(\bar{\lambda}), u(\bar{\lambda}))$ and $(y, u) = (y(\lambda), u(\lambda))$ for notational simplicity.

Let $K \in \mathcal{B}(\bar{\lambda})$ be arbitarily chosen. From \cref{KKT-general-augmented-reduced}, there exists $\bar{\xi} \in \widehat{M}(\bar{\lambda})$ with ${\rm supp}(\bar{\xi}) \subseteq K \subseteq I(\bar{u})$ such that $(\bar{y}, \bar{u}, \bar{\xi})$ satisfies
\begin{equation}
\bar{y} = b - \widehat{A}_{:K}\bar{\xi}_K, \quad
BA^T \bar{y} = \widehat{A}^T\bar{y} = \bar{\lambda} B\bar{u}, \quad
B_{K:}\bar{u} = d_K, \quad \xi_{K^c} = 0,
\end{equation}
where $K^c$ is the complement of $K$, which implies
\[
\bar{\lambda} d_K = \widehat{A}_{:K}^T(b - \widehat{A}_{:K}\bar{\xi}_K).
\]
Since $\widehat{A}_{:K}$ is of full column rank, we have
\[
\bar{\xi}_K = (\widehat{A}_{:K}^T\widehat{A}_{:K})^{-1}(\widehat{A}_{:K}^Tb - \bar{\lambda} d_K).\]
Consequently, we have
		\[
		\bar{y} = P^K b + \bar{\lambda} h^K,
		\]
		where $h^K = \widehat{A}_{:K}(\widehat{A}_{:K}^T\widehat{A}_{:K})^{-1}d_K \in \mathcal{H}(\bar{\lambda})$.

(i) It follows from \cref{prop: strict-monotonicity-tame} that $u(\cdot)$ is locally Lipshitz continuous, which implies that $I(u) \subseteq I(\bar{u})$. Next, we prove that $\mathcal{B}(\lambda) \subseteq \mathcal{B}(\bar{\lambda})$. If not, then there exists a sequence $\{\lambda_k\}_{k \geq 1} \subseteq \mathcal{N}(\bar{\lambda})$ converges to $\bar{\lambda}$, such that for all $k$, there is an index set $K^k \in \mathcal{B}(\lambda_k) \backslash \mathcal{B}(\bar{\lambda})$. Denote the solution to \cref{eq: dual-LS-reg} with the parameter $\lambda_k$ as $(y^k, u^k)$. Since there exist only finitely many choices for the index sets in $\mathcal{B}(\cdot)$, if necessary by taking a subsequence we assume that the index sets $K^k$ are identical for all $k \geq 1$. Denote the common index set as $\tilde{K}$. Then, the matrix $\widehat{A}_{:\tilde{K}}$ has full column rank and there exists $\xi^k \in \widehat{M}(\lambda_k)$ (and $(B^T\xi^k, \xi^k) \in M(\lambda_k)$) such that ${\rm supp}(\xi^k) \subseteq \tilde{K} \subseteq I(u^k)$ but $\tilde{K} \not\in \mathcal{B}(\bar{\lambda})$. Since $I(u^k) \subseteq I(\bar{u})$, then there is no $\xi \in \mathcal{B}(\bar{\lambda})$ such that ${\rm supp}(\xi) \subseteq \tilde{K}$. However, since $\xi^k \in \widehat{M}(\lambda^k)$, it satisfies
\begin{equation}
y^k - b + \widehat{A}_{:\tilde{K}}\xi^k_{\tilde{K}} = 0.
\end{equation}
As $y(\cdot)$ is locally Lipschitz continuous and $\widehat{A}_{:\tilde{K}}$ is of full column rank, the sequence $\{\xi^k\}_{k \geq 1}$ is bounded. Let $\tilde{\xi}$ be an accumulation point of $\{\xi^k\}_{k \geq 1}$, then $\tilde{\xi} \in \mathcal{B}(\bar{\lambda})$ and ${\rm supp}(\tilde{\xi}) \subseteq \tilde{K}$. This is a contradiction. Therefore, $\mathcal{B}(\lambda) \subseteq \mathcal{B}(\bar{\lambda})$. From the definition of $\mathcal{H}(\cdot)$ in \cref{eq: Jacobian-y-lasso}, we also have $\mathcal{H}(\lambda) \subseteq \mathcal{H}(\bar{\lambda})$.

(ii) Let $K \in \mathcal{B}(\lambda)$ be arbitarily chosen. It follows from \cref{eq: expression-y} that
		\[
		y = P^K b + \lambda h^K,
		\]
		where $h^K = \widehat{A}_{:K}(\widehat{A}_{:K}^T\widehat{A}_{:K})^{-1}d_K \in \mathcal{H}(\lambda)$.
  Since $K \in \mathcal{B}(\lambda) \subseteq \mathcal{B}(\bar{\lambda})$ and $h^K \in \mathcal{H}(\lambda) \subseteq \mathcal{H}(\bar{\lambda})$, in a same vein, we have
  \[
    \bar{y} = P^K b + \bar{\lambda} h^K.
  \]
  As a result, for all $h \in \mathcal{H}(\lambda)$,
  \[
  y = \bar{y} + (\lambda - \bar{\lambda})h.
  \]
  We complete the proof of the lemma.
\end{proof}

Next, we prove the nondegeneracy of $\partial\varphi(\bar{\lambda})$ for any $\bar{\lambda} \in (0, \lambda_{\infty})$, which is important for analyzing the convergence rates of the secant method for solving \cref{eq: root-finding-levelset-lst}.

\begin{theorem}
\label{prop: property-H-gamma}
Let $p(\cdot)$ be a polyhedral gauge function. Assume that $0 < \lambda_{\infty} < +\infty$. For any $\bar{\lambda} \in (0, \lambda_{\infty})$, it holds that
\begin{romannum}
\item for any integer $k\ge 1$, the function     $\varphi(\cdot)$ is piecewise $C^k  $ in an open interval containing $\bar{\lambda}$;
\item all $v \in \partial\varphi(\bar{\lambda})$ are positive.
\end{romannum}

\end{theorem}
\begin{proof}
Choose a sufficiently small $\varsigma > 0$ such that $\mathcal{N}(\bar{\lambda}) = (\bar{\lambda} - \varsigma, \bar{\lambda} + \varsigma) \subseteq (0, \lambda_{\infty})$ and $\mathcal{B}(\lambda) \subseteq \mathcal{B}(\bar{\lambda})$ for any $\lambda \in \mathcal{N}(\bar{\lambda})$. Let $\lambda \in \mathcal{N}(\bar{\lambda}) \backslash \bar{\lambda}$ and $K \in \mathcal{B}(\lambda)$ be arbitrarily chosen. Denote
\[
h^K = \widehat{A}_{:K}(\widehat{A}_{:K}^T\widehat{A}_{:K})^{-1}d_K.
\]
Then, we have $h^K \in \mathcal{H}(\lambda) \subseteq \mathcal{H}(\bar{\lambda})$.

Now, we prove Part (i) of the theorem. From the fact
\[\langle P^K b, h^K \rangle = 0\]
and \cref{lemma: upper-semicontinuouity}, we know that
  \[
  \varphi(\lambda) = \sqrt{\|P^K b \|^2 + \lambda^2\|h^K\|^2} \quad \mbox{ and } \quad \varphi(\bar{\lambda}) = \sqrt{\|P^K b \|^2 + (\bar{\lambda})^2\|h^K\|^2}.
  \]
  Define $\varphi^K: \mathbb{R} \to \mathbb{R}_+$ by
  \begin{equation}
   \label{def: varphi_K}
   \varphi^K(s) := \sqrt{\|P^Kb\|^2 + s^2\|h^K\|^2}, \quad   s \in \mathbb{R}.
  \end{equation}
  From \cref{prop: strict-monotonicity-tame}, we know that $\varphi(\cdot)$ is strictly increasing on $(0, \lambda_{\infty}]$. Therefore, it holds that
  \begin{equation}
      \varphi^K(\lambda) = \varphi(\lambda) \neq \varphi(\bar{\lambda}) = \varphi^K(\bar{\lambda}),
  \end{equation}
  which implies that $h^K \neq 0$. Since $K \in \mathcal{B}(\lambda)$ is arbitrarily chosen, we obtain that
  \begin{equation}
  \label{eq: nonzero-hk}
      h^K \neq 0, \, \forall K \in \mathcal{B}(\lambda).
  \end{equation}
  Thus, for any integer $k \geq 1$, $\varphi^K(\cdot)$ is $C^k$ on $\mathcal{N}(\bar{\lambda})$.

Denote $\bar{\mathcal{B}} = \bigcup_{\lambda \in \mathcal{N}(\bar{\lambda}) \backslash \bar{\lambda}} \mathcal{B}(\lambda)$. We know that $\bar{\mathcal{B}} \subseteq \mathcal{B}(\bar{\lambda})$ and $|\bar{\mathcal{B}}|$ is finite. Moreover,
  \begin{equation}
      \varphi(s) \in \{\varphi^K(s)\}_{K \in \bar{\mathcal{B}}}, \quad  \forall \, s \in \mathcal{N}(\bar{\lambda}),
  \end{equation}
  which implies that for any $k \geq 1$, $\varphi(\cdot)$ is piecewise $C^k$ on $\mathcal{N}(\bar{\lambda})$.

  Next, we prove Part (ii) of the theorem. It follows from \cref{eq: nonzero-hk} that for any $K \in \bar{\mathcal{B}}$ and $\lambda > 0$,
  \begin{equation}
      h^K \neq 0 \quad  \mbox{ and } \quad (\varphi^K)'(\lambda) = \lambda\|h^K\|^2/\varphi^K(\lambda) > 0.
  \end{equation}
  By \cite[Theorem 2.5.1]{clarke1983optimization} and the upper semicontinuity of $\partial_{\rm B}\varphi(\cdot)$, we have
  \begin{equation}
      \partial\varphi(\bar{\lambda}) \subseteq {\rm conv}(\{\bar{\lambda}\|h^K\|^2/\varphi(\bar{\lambda}) \,|\, K \in \bar{\mathcal{B}}\}),
  \end{equation}
which implies
\[
v > 0, \quad \forall \, v \in \partial\varphi(\bar{\lambda}).
\]
We complete the proof of the theorem.
\end{proof}

\section{A secant method for \cref{eq: main-prob}}
\label{sec: secant-levelset}
In this section, we will design a fast convergent secant method for solving \cref{eq: root-finding-levelset-lst} and prove its convergence rates.

\subsection{A fast convergent secant method for semismooth equations}

Let $f: \mathbb{R} \to \mathbb{R}$ be a locally Lipschitz continuous function which is semismooth at a solution $x^*$ to the following equation
\begin{equation}
\label{eq: F-equation}
    f(x) = 0.
\end{equation}
In this section, we analyze the convergence of the secant method described in \cref{alg: classic-secant} with two generic starting points $x^{-1}$ and $x^0$.

\begin{algorithm}
	\caption{A secant method for solving \cref{eq: F-equation}}
	\label{alg: classic-secant}
	\begin{algorithmic}[1]
		\STATE{\textbf{Input}: $x^{-1}, x^0 \in \mathbb{R}$.}
		\STATE{\textbf{Initialization}: Set $k = 0$.}
		\WHILE{$f(x^k) \neq 0$}
        \STATE{\textbf{Step 1}. Compute
        \begin{equation}
            x^{k+1} = x^k - (\delta_f(x^k, x^{k-1}))^{-1} f(x^k).
        \end{equation}}
        \STATE{\textbf{Step 2}. Set $k = k+1$.}
		\ENDWHILE
		\STATE{\textbf{Output}: $x^k$}.
	\end{algorithmic}
\end{algorithm}

The convergence results of \cref{alg: classic-secant} are given in the following proposition. The proof can be obtained by following the procedure in the proof of \cite[Theorem 3.2]{Potra98}.
\begin{proposition}
\label{thm: convergence-classic-secant}
Suppose that $f: \mathbb{R} \to \mathbb{R}$ is semismooth at a solution $x^*$ to \cref{eq: F-equation}. Let $d^-$ and $d^+$ be the lateral derivatives of $f$ at $x^*$ as defined in \cref{def: lateral-derivative}. If $d^-$ and $d^{+}$ are both positive (or negative), then there are two neighborhoods $\mathcal{U}$ and $\mathcal{N}$ of $x^*$, $\mathcal{U} \subseteq \mathcal{N}$, such that for $x^{-1}, x^0 \in \mathcal{U}$, \cref{alg: classic-secant} is well defined and produces a sequence of iterates $\{x^k\}$ such that $\{x^k\} \subseteq \mathcal{N}$. The sequence $\{x^k\}$ converges to $x^*$ 3-step Q-superlinearly, i.e., $|x^{k+3} - x^*| = o(|x^k - x^*|)$. Moreover, it holds that
\begin{romannum}
    \item $|x^{k+1} - x^*| \leq \frac{|d^+ - d^{-} + o(1)|}{\min\{|d^+|, |d^{-}|\} + o(1)} |x^k - x^*|$ for $k \geq 0$;
    \item if $\alpha := \frac{|d^+ - d^-|}{\min\{|d^+|, |d^-|\}} < 1$, then $\{x^k\}$ converges to $x^*$ Q-linearly with Q-factor $\alpha$;
    \item if $f$ is $\gamma$-order semismooth at $x^*$ for some $\gamma > 0$, then $|x^{k+3} - x^*| = O(|x^k - x^*|^{1 + \gamma})$ for sufficiently large $k$; the sequence $\{x^k\}$ converges to $x^*$ 3-step quadratically if $f$ is strongly semismooth at $x^*$.
\end{romannum}
\end{proposition}

Here, we only consider the case for $d^+ \cdot d^- > 0$ since the function $\varphi(\cdot)$ we are interested in is nondecreasing. For the case $d^+ \cdot d^- < 0$, one can refer to \cite[Theorem 3.3]{Potra98}.

When $|d^+ - d^{-}|$ is small and $f$ is strongly semimsooth, we know from \cref{thm: convergence-classic-secant} that the secant method converges with a fast linear rate and 3-step Q-quadratic rate. We provide a numerical example slightly modified from \cite[Equation (3.15)]{Potra98} to illustrate the convergence rates obtained in \cref{thm: convergence-classic-secant}. We test \cref{alg: classic-secant} with $x^{-1} = 0.01$ and $x^{0} = 0.005$ for finding the zero $x^* = 0$ of
\begin{equation}
\label{test-example-Qlinear}
f(x) = \left\{
\begin{array}{ll}
x(x+1) & \mbox{if $x < 0$},\\
-\beta x(x-1) & \mbox{if $x \geq 0$},
\end{array}
\right.
\end{equation}
where $\beta$ is chosen from $\{1.1, 1.5, 2.1\}$. The numerical results are shown in \cref{tab: 1d-Qlinear-verification}, which coincide with our theoretical results.

\begin{table}[tbhp]
{\footnotesize
\caption{The numerical performance of finding the zero of \cref{test-example-Qlinear}. Case I: $\beta = 1.1, \; d^+ = 1.1, \; d^- = 1, \mbox{ and } \alpha = 0.1$; Case II: $\beta = 1.5, \; d^+ = 1.5, \; d^- = 1, \mbox{ and } \alpha = 0.5$; Case III: $\beta = 2.1, \; d^+ = 2.1, \; d^- = 1, \mbox{ and } \alpha = 1.1$.}\label{tab: 1d-Qlinear-verification}
\begin{center}
\vspace{-3mm}
\begin{tabular}{cc|cccccccc} \hline
Case & Iter & 1 & 2 & 3 & 4 & 5&6&7&8 \\ \hline
I & $x$ & -5.1e-5 & -4.3e-6 & 2.2e-10 & -2.2e-11 & -1.8e-12 & 4.1e-23 & -4.1e-24 & -3.4e-25\\
\hline
II & $x$ & -5.1e-5 & -1.7e-5 & 8.4e-10 & -4.2e-10 & -1.1e-10 & 4.5e-20 & -2.2e-20 & -5.6e-21\\
\hline
III & $x$ & -5.1e-5 & -2.6e-5 & 1.3e-9 & -1.5e-9 & -5.1e-10 & 7.4e-19 & -8.2e-19 & -2.8e-19\\
\hline
\end{tabular}
\end{center}
}
\end{table}

Note that \cite[Lemma 4.1]{Potra98} implies that the sequence $\{x^k\}$ generated by \cref{alg: classic-secant} converges suplinearly with R-order at least $\sqrt[3]{2}$. Next, we will prove that the sequence $\{x^k\}$ generated by \cref{alg: classic-secant} converges superlinearly to a solution $x^*$ to \cref{eq: F-equation} with {R}-order at least $(1 + \sqrt{5})/2$  when $f$ is strongly semismooth at $x^*$ and $\partial f(x^*)$ is a singleton and  nondegenerate.
\begin{proposition}
\label{thm: R-order-secant}
    Let $x^*$ be a solution to \cref{eq: F-equation}. {Let $\{x^k\}$ be the sequence generated by \cref{alg: classic-secant} for solving \eqref{eq: F-equation}. For $k \geq -1$, denote $e_k := x^k - x^*$ and assume that $|e_k| > 0$. For $k \geq -1$, denote $c_k := |e_k|/(|e_{k-1}||e_{k-2}|)$.
    } Assume that $\partial f(x^*)$ is a singleton and nondegenerate. It holds that
    \begin{romannum}
        \item if $f$ is semismooth at $x^*$, the sequence $\{x^k\}$ converges to $x^*$ Q-superlinearly;
        \item if $f$ is strongly semismooth at $x^*$, then either one of the following two properties is satisfied:
        {(1) $\{x^k\}$ converges to $x^*$ superlinearly with Q-order at least $(1 + \sqrt{5})/2$};
        {(2) $\{x^k\}$ converges to $x^*$ superlinearly with R-order at least $(1 + \sqrt{5})/2$ and for any  constant $\underline{C} > 0$, there exists a subsequence $\{c_{i_k}\}$ satisfying $c_{i_k} < \underline{C}i_k^{-i_k}$}.
    \end{romannum}
\end{proposition}

\begin{proof}
Let $\mathcal{N}$ and $\mathcal{U}$ be the neighborhoods of $x^*$ specified in \cref{thm: convergence-classic-secant}. Assume that $x^{-1}, x^0 \in \mathcal{U}$. Then \cref{alg: classic-secant} is well defined and it generates a sequence $\{x^k\} \subseteq \mathcal{N}$ which converges to $x^*$. Denote $\partial f(x^*) = \{v^*\}$ for some $v^* \neq 0$. Let $d^-$ and $d^+$ be the lateral derivatives of $f$ at $x^*$ as defined in \cref{def: lateral-derivative}. Then,
\[
d^+ = d^- = v^*.
\]
 Let $K_1$ be a sufficiently large integer. For all $k \geq K_1$, we have
\begin{equation}
    e_{k+1} = \delta_f(x^k, x^{k-1})^{-1}[\delta_f(x^k, x^{k-1}) - \delta_f(x^k, x^*)]e_k.
\end{equation}

(i) Assume that $f$ is semismooth at $x^*$. We estimate
$$\delta_f(x^k, x^{k-1})^{-1}[\delta_f(x^k, x^{k-1}) - \delta_f(x^k, x^*)]$$
by considering the following two cases.

(i-a) $x^k, x^{k-1} > x^*$ or $x^k, x^{k-1} < x^*$. From \cref{lemma: lateral-derivative-approx}, we obtain that
\begin{eqnarray*}
    |e_{k+1}| &=& |\delta_f(x^k, x^{k-1})^{-1}[\delta_f(x^k, x^{k-1}) - \delta_f(x^k, x^*)]e_k|\\
    &=& |(v^* + o(1))^{-1}[(v^* + o(1)) - (v^* + o(1))]e_k|\\
    &=& o(|e_k|).
\end{eqnarray*}

(i-b) $x^{k-1} < x^* < x^k$ or $x^k < x^* < x^{k-1}$. We will consider the first case. The second case can be treated similarly. By \cref{lemma: lateral-derivative-approx}, it holds that
\begin{eqnarray*}
    \delta_f(x^k, x^{k-1}) &=& \frac{f(x^k) - f(x^*) + f(x^*) - f(x^{k-1})}{x^{k} - x^{k-1}}\\
    &=& \frac{(v^* e_{k} + o(|e_k|)) - (v^* e_{k-1} + o(|e_{k-1}|))}{x^k - x^{k-1}}\\
    &=& v^* + o(1).
\end{eqnarray*}
Therefore,
\begin{eqnarray*}
    |e_{k+1}| &=& |\delta_f(x^k, x^{k-1})^{-1}[\delta_f(x^k, x^{k-1}) - \delta_f(x^k, x^*)]e_k|\\
    &=& |(v^* + o(1))^{-1}||(v^* + o(1)) - (v^* + o(1))||e_k|\\
    &=& o(|e_k|).
\end{eqnarray*}
Thus, we prove that the sequence $\{x^k\}$ converges to $x^*$ Q-superlinearly.

(ii) Now, assume that $f$ is strongly semismooth at $x^*$. We build the recursion for $e_k$ for sufficiently large integers $k$ by considering the following two cases.

(ii-a) $x^k, x^{k-1} > x^*$ or $x^k, x^{k-1} < x^*$. From \cref{lemma: lateral-derivative-approx}, we obtain
\begin{eqnarray*}
   |e_{k+1}| &=& |\delta_f(x^k, x^{k-1})^{-1}[\delta_f(x^k, x^{k-1}) - \delta_f(x^k, x^*)]e_k|\\
   &=& |v^* + O(|e_k| + |e_{k-1}|)|^{-1}|(O(|e_k| + |e_{k-1}|))||e_k|\\
   &=& O(|e_k|(|e_k| + |e_{k-1}|)).
\end{eqnarray*}

(ii-b) $x^{k-1} < x^* < x^k$ or $x^k < x^* < x^{k-1}$. We will consider the first case. The second case can be treated similarly. By \cref{lemma: lateral-derivative-approx}, it holds that
\begin{eqnarray*}
    \delta_f(x^k, x^{k-1}) &=& \frac{f(x^k) - f(x^*) + f(x^*) - f(x^{k-1})}{x^{k} - x^{k-1}}\\
    &=& \frac{(v^* e_{k} + O(|e_k|^2)) - (v^* e_{k-1} + O(|e_{k-1}|^2))}{x^k - x^{k-1}}\\
    &=& v^* + O(|e_k| + |e_{k-1}|)
\end{eqnarray*}
and
\begin{eqnarray*}
    |e_{k+1}| &=& |\delta_f(x^k, x^{k-1})^{-1}[\delta_f(x^k, x^{k-1}) - \delta_f(x^k, x^*)]e_k|\\
    &=& |v^* + O(|e_k| + |e_{k-1}|)|^{-1}|O(|e_k| + |e_{k-1}|)||e_k|\\
    &=& O(|e_k|(|e_k| + |e_{k-1}|)).
\end{eqnarray*}
Therefore, for {sufficiently large integers $k$}, we have
\begin{equation}
\label{eq: recursion-bound}
|e_{k+1}| = O(|e_k|^2 + |e_k||e_{k-1}|)
\end{equation}
and
\begin{equation}
\label{eq: bound-limek}
{\limsup_{k \to \infty} c_k = O(1/|v^*|) < +\infty}.
\end{equation}
{Then, there exists a constant $\bar{C} > 0$ and a positive integer $K_2$ such that
\[
|e_{k+1}| \leq \bar{C} |e_k||e_{k-1}|, \quad \forall \, k  \geq K_2.
\]
Therefore, it follows from \cite[Theorem 9.2.9]{Ortega70} that $\{x^k\}$ converges to $x^*$ with R-order at least $(1 + \sqrt{5})/2$.

If there exists a constant $\underline{C} > 0$ such that $c_k \geq \underline{C} k^{-k}$ for all $k$ sufficiently large, it follows from \cite[Corollary 3.1]{FA-Potra89} that $\{x^k\}$ converges to $x^*$ Q-superlinearly with Q-order at least $(1 + \sqrt{5})/2$. We complete the proof.
}
\end{proof}

\begin{proposition}\label{prop:polyhedral-gauge-Q-order}
Let $p(\cdot)$ be a polyhedral gauge function and $\lambda^*$ be the solution to \eqref{eq: root-finding-levelset-lst}. Assume that $0 < \lambda_{\infty} < +\infty$. If $\partial \varphi(\lambda^*)$ is a singleton, the sequence $\{\lambda_k\}$ generated by \cref{alg: classic-secant} for solving \eqref{eq: root-finding-levelset-lst} converges to $\lambda^*$ Q-superlinearly with Q-order at least $(1 + \sqrt{5})/2$.
\end{proposition}

\begin{proof}
The assumption $\partial\varphi(\lambda^*)$ is a singleton implies that $\varphi(\cdot)$ is strictly differentiable at $\lambda^*$ \cite[Proposition 2.2.4]{clarke1983optimization}. It follows from \cref{prop: property-H-gamma} that $\varphi'(\lambda^*) > 0$ and $\varphi(\cdot)$ is piecewise $C^k$ for any positive integer $k \geq 1$ in a neighborhood of $\lambda^*$.

Choose a sufficiently small $\varsigma > 0$ such that $\mathcal{N}(\lambda^*) := (\lambda^* - \varsigma, \lambda^* + \varsigma) \subseteq (0, \lambda_{\infty})$ and $\mathcal{B}(\lambda) \subseteq \mathcal{B}(\lambda^*)$ for any $\lambda \in \mathcal{N}(\lambda^*)$. Denote $\bar{\mathcal{B}} = \bigcup_{\lambda \in \mathcal{N}(\lambda^*)\backslash \lambda^*}\mathcal{B}(\lambda)$. Let $K \in \bar{\mathcal{B}}$ be arbitrarily chosen. Define
$\varphi^K : \mathbb{R} \to \mathbb{R}_+$ by
\[
\varphi^K(s) := \sqrt{\|P^K b\|^2 + s^2\|h^K\|^2}, \quad s \in \mathbb{R}.
\]
By choosing a smaller $\varsigma$ if necessary, we assume that $\{\varphi^K\}_{K \in \bar{\mathcal{B}}}$ is a minimal local representation for $\varphi(\cdot)$ at $\lambda^*$. Therefore, it follows from \cite[Theorem 2]{qi2007almost} that
\[
\varphi'(\lambda^*) = (\varphi^K)'(\lambda^*) = \lambda^*\|h^K\|^2/\varphi(\lambda^*), \quad \forall K \in \bar{\mathcal{B}}.
\]
Since $\varphi'(\lambda^*) > 0$, we have
\[
\|h^K\| = \|h^{K'}\|, ~ \|P^Kb\| = \|P^{K'}b\|, \quad \forall K, K' \in \bar{\mathcal{B}}.
\]
Therefore, $\varphi(\cdot)$ is $C^k$ on $\mathcal{N}(\lambda^*)$ for any integer $k \geq 1$. It follows from \cite[Example 6.1]{Traub64} that $\{\lambda_k\}$ converges to $\lambda^*$ Q-superlinearly with Q-order at least $(1 + \sqrt{5})/2$.
\end{proof}

We give the following example to show that a function satisfying the assumptions in (ii) of \cref{thm: R-order-secant} is not necessarily piecewise smooth:
\begin{equation}
\label{eq: 1d-construction}
f(x) = \left\{
\begin{array}{ll}
\kappa x, & \mbox{if $x < 0$},\\[5pt]
-\frac{1}{3}\left(\frac{1}{4^k}\right) + (1 + \frac{1}{2^k})x, & \mbox{if $x \in \left[\frac{1}{2^{k+1}}, \frac{1}{2^k}\right] \quad \forall k \geq 0$}, \\[5pt]
2x - \frac{1}{3} & \mbox{if $x > 1$},
\end{array}
\right.
\end{equation}
where $\kappa$ is a given constant.

\begin{proposition}
The function $f$ defined in \cref{eq: 1d-construction} is strongly semismooth at $x = 0$ but not piecewise smooth in the neighborhood of $x = 0$.
\end{proposition}

\begin{proof}
By the construction of $f(\cdot)$, we know that it is not piecewise smooth in the neighborhood of $x = 0$ since there are infinitely many non-differentiable points. Next, we show that $f$ is strongly semismooth at $x = 0$.

Firstly, it is not difficult to verify that $f$ is Lipschitz continuous with modulus $L = \max\{|\kappa|, ~ 2\}$. Secondly, we know that $f'(0; -1) = \kappa$, and for any integer $k \geq 0$, it holds that
\[
1 + \frac{1}{3 \times 2^k}\leq f(x)/x \leq \left(1 + \frac{1}{3 \times 2^{k-1}}\right) \quad \forall \, x \in \left[\frac{1}{2^{k+1}}, \frac{1}{2^{k}}\right],
\]
which implies that
\[
f'(x; 1) = \lim_{x \downarrow 0} f(x)/{x} = 1.
\]
Therefore, $f$ is directionally differentiable at $x = 0$. Note that both $\partial f(0)$ and $\partial_{\rm B}f(0)$ are singleton when $\kappa = 1$.

Next, we show that $f(\cdot)$ is strongly G-semismooth at $x = 0$. On the one hand, for any $x < 0$, we have
\begin{eqnarray*}
|f(x) - f(0) - \kappa x| &= |\kappa x - \kappa x| &= 0.
\end{eqnarray*}
On the other hand, for any integer $k \geq 1$ and $x \in \left[\frac{1}{2^{k+1}}, \frac{1}{2^k}\right]$, we know
\[
1 + 2^{-k} \leq |v| \leq 1 + 2^{1 - k}, \quad \forall \, v \in \partial f(x),
\]
which implies that
\begin{equation*}
\left|f(x) - f(0) - v(x) x \right| = \left|-\frac{1}{3}\left(\frac{1}{4^k}\right) + \left(1 + \frac{1}{2^k}\right)x - v(x)x\right| \leq \frac{1}{2^{k-1}}x \leq 4 x^2.
\end{equation*}
Therefore,
\[
|f(x) - f(0) - v(x)x| = O(|x|^2), \quad \forall \, x \to 0.
\]
The proof of the proposition is completed.
\end{proof}

We end this subsection by illustrating the numerical performance of \cref{alg: classic-secant} for finding the root of $f(x)$ given in \cref{eq: 1d-construction} with $\kappa = 1$. Note that $x^* = 0$ is the unique solution. In \cref{alg: classic-secant}, we choose $x^0 = 0.5$ and $x^{-1} = x^0 + 0.1 \times f(0.5)^2 = 0.545$. The numerical results are shown in \cref{tab: 1d-verification}.
\begin{table}[tbhp]
{\footnotesize
\caption{The numerical performance of \cref{alg: classic-secant} on finding the zero of \cref{eq: 1d-construction}.}\label{tab: 1d-verification}
\begin{center}
\vspace{-3mm}
\begin{tabular}{c|cccccccc} \hline
Iter &  1 & 2 & 3 & 4 & 5&6&7&8 \\ \hline
$x$ & 1.7e-1 & 3.6e-2 & 4.0e-3 & 1.0e-4 & 2.7e-7 & 2.0e-11 & 4.0e-18 & 6.1e-29\\
$f(x)$ &	1.9e-1 & 	3.7e-2	& 4.0e-3 &	1.0e-4	& 2.7e-7 & 2.0e-11 & 4.0e-18 & 6.1e-29\\ \hline
\end{tabular}
\end{center}
}
\end{table}
We can observe that the generated sequence $\{x_k\}$ converges to the solution $x^* = 0$ superlinearly with R-order at least $(1+\sqrt{5})/2$.

\subsection{A globally convergent secant method for \cref{eq: main-prob}}
In this section, we propose a globally convergent secant method for solving \cref{eq: main-prob} via finding the root of \cref{eq: root-finding-levelset-lst}. The algorithm is described in \cref{alg:hybrid-levelset}.

\begin{algorithm}
	\caption{A globally convergent secant method for \cref{eq: main-prob}}
	\label{alg:hybrid-levelset}
	\begin{algorithmic}[1]
		\STATE{\textbf{Input}: $A \in \mathbb{R}^{m \times n}$, $b \in \mathbb{R}^n$, $\mu \in (0,1)$, $\lambda_{-1}, \lambda_{0}, \lambda_1$ in $(0, \lambda_{\infty})$ satisfying $\varphi(\lambda_0) > \varrho$, and $\varphi(\lambda_{-1}) < \varrho$.}
		\STATE{\textbf{Initialization}: Set $i = 0$, $\underline{\lambda} = \lambda_{-1}$, and $\overline{\lambda} = \lambda_0$.}
		\FOR{$k = 1, 2, \dots$}
		\STATE{Compute
			\begin{equation}\label{eq:secant_steps}
			\hat{\lambda}_{k+1} = \lambda_k - \frac{\lambda_{k} - \lambda_{k-1}}{\varphi(\lambda_k) - \varphi(\lambda_{k-1})}(\varphi(\lambda_k) - \varrho).
			\end{equation}
		}
		\IF{$\hat{\lambda}_{k+1} \in [\lambda_{-1}, \lambda_0]$}
		\STATE{Compute $x(\hat{\lambda}_{k+1})$ and $\varphi(\hat{\lambda}_{k+1})$. Set $i = i+1$.}
		\STATE{\textbf{if} either (i) or (ii) holds: (i) $i \geq 3$ and $|\varphi(\hat{\lambda}_{k+1}) - \varrho| \leq \mu |\varphi(\lambda_{k-2}) - \varrho|$ (ii) $i < 3$, \textbf{then} set $\lambda_{k+1} = \hat{\lambda}_{k+1}$, $x(\lambda_{k+1}) = x(\hat{\lambda}_{k+1})$; \textbf{else} go to line 9.}
		\ELSE
		\STATE{\textbf{if} $\varphi(\hat{\lambda}_{k+1}) > \varrho$, \textbf{then} set $\overline{\lambda} = \min\{\overline{\lambda}, \hat{\lambda}_{k+1}\}$; \textbf{else} set $\underline{\lambda} = \max\{\underline{\lambda}, \hat{\lambda}_{k+1}\}$}.
		\STATE{Set $\lambda_{k+1} = 1/2(\overline{\lambda} + \underline{\lambda})$. Compute $x(\lambda_{k+1})$ and $\varphi(\lambda_{k+1})$. Set $i = 0$.}
		\ENDIF
		\STATE{\textbf{if} $\varphi(\lambda_{k+1}) > \varrho$, \textbf{then} set $\overline{\lambda} = \min\{\overline{\lambda}, \lambda_{k+1}\}$; \textbf{else} set $\underline{\lambda} = \max\{\underline{\lambda}, \lambda_{k+1}\}$}.
		\ENDFOR
		\STATE{\textbf{Output}: $x(\lambda_k)$ and $\lambda_k$}.
	\end{algorithmic}
\end{algorithm}

\begin{theorem}
\label{thm: convergence-hybrid-secant}
 Let $p(\cdot)$ be a gauge function and assume that $0 < \lambda_{\infty} < +\infty$. Denote $\lambda^*$ as the solution to \cref{eq: root-finding-levelset-lst}. {Then} \cref{alg:hybrid-levelset} is well defined and the sequences $\{\lambda_k\}$ and $\{x(\lambda_k)\}$ converge to $\lambda^*$ and a solution $x(\lambda^*)$ to \cref{eq: main-prob}, respectively. Denote $e_k = \lambda_k - \lambda^*$ for all $k \geq 1$. Suppose that both $d^+$ and $d^-$ of $\varphi(\cdot)$ at $\lambda^*$ as defined in \cref{def: lateral-derivative} are positive, the following properties hold for all sufficiently large integer $k$:
\begin{romannum}
\item If $\varphi(\cdot)$ is semismooth at $\lambda^*$, then $|e_{k+3}| = o(|e_k|)$;
\item if $\varphi(\cdot)$ is $\gamma$-order semismooth at $\lambda^*$ for some $\gamma > 0$, then $|e_{k+3}| = O(|e_k|^{1 +\gamma})$;
\item if $\partial\varphi(\lambda^*)$ is a singleton and $\varphi(\cdot)$ is semismooth at $\lambda^*$, then $\{e_{k}\}$ converges to zero Q-superlinearly; if {$p(\cdot)$ is further assumed to be polyhedral and} $\partial\varphi(\lambda^*)$ is a singleton, then $\{e_{k}\}$ converges to zero superlinearly with Q-order $(1 + \sqrt{5})/2$.
\end{romannum}
\end{theorem}

\begin{proof}
When $p(\cdot)$ is a gauge function, we know from \cref{prop: strict-monotonicity-tame} that $\varphi(\cdot)$ is strictly increasing on $(0, \lambda_{\infty}]$, which implies that the sequences $\{\hat{\lambda}_k\}$ and $\{\lambda_k\}$ generated in \cref{alg:hybrid-levelset} are well defined. For any $k \geq 1$, if we run the algorithm for three more iterations, then it holds that either (a) $(\overline{\lambda} - \underline{\lambda})$ will reduce at least half; or (b) $|\varphi(\lambda_{k+3}) - \varrho| \leq \mu |\varphi(\lambda_k) - \varrho|$. Therefore, the sequence $\{\lambda_k\}$ will converge to $\lambda^*$. Suppose that $\varphi(\cdot)$ is semismooth at $\lambda^*$ and both $d^+$ and $d^{-}$ are positive. We know from \cref{thm: convergence-classic-secant} that there exists a positive integer $k_{\rm max}$ such that for all $k \geq k_{\rm max}$, $\hat{\lambda}_k \in [\lambda_{-1}, \lambda_{0}]$ and
\begin{equation}
\label{eq: 3-steps-superlinear}
    |\hat{\lambda}_{k+3} - \lambda^*| = o(|\lambda_k - x^*|).
\end{equation}
Therefore, it follows from \cref{lemma: lateral-derivative-approx} that
\[
\frac{|\varphi(\hat{\lambda}_{k+3}) - \varrho|}{|\varphi(\lambda_k) - \varrho|} = \frac{\delta_\varphi(\hat{\lambda}_{k+3}, \lambda^*)}{\delta_\varphi(\lambda_k, \lambda^*)} \times \frac{|\hat{\lambda}_{k+3} - \lambda^*|}{|\lambda_k - \lambda^*|} \leq \frac{\max\{d^+, d^-\} + o(1)}{\min\{d^+, d^-\} + o(1)} \times o(1).
\]
Thus, for all $k \geq k_{\rm max}$,
\[
|\varphi(\hat{\lambda}_{k+3}) - \varrho| \leq \mu |\varphi(\lambda_{k}) - \varrho|.
\]
The rest of the proof of this theorem follows from \cref{thm: convergence-classic-secant} {and \cref{prop:polyhedral-gauge-Q-order}}.
\end{proof}

To better illustrate the efficiency of \cref{alg:hybrid-levelset}, we will compare its performance to the HS-Jacobian based semismooth Newton method for solving \cref{eq: root-finding-levelset-lst} on the least-squares constrained Lasso problem, {where the HS-Jacobian is available}. The following proposition shows that for the least-squares constrained Lasso problem, $\partial_{\rm HS}\varphi(\bar{\lambda})$ is positive for any $\bar{\lambda} \in (0, \lambda_{\infty})$.

\begin{proposition} \label{eq:HS-Jacobian-ell1}
Suppose that $p(\cdot)$ is a polyhedral gauge function and $\partial p(0)$ has the expression as in \cref{eq: representation-p-polar}. Assume that $0 < \lambda_{\infty} < + \infty$ and let $\bar{\lambda} \in (0, \lambda_{\infty})$ be arbitrarily chosen. Let $\mathcal{B}(\bar{\lambda})$ and $\mathcal{V}(\bar{\lambda})$ be the sets defined as in \cref{B-set-lasso} and \cref{eq: Jacobian-varphi} for $\lambda = \bar{\lambda}$. If $d_K \neq 0$ for all $K \in \mathcal{B}(\bar{\lambda})$, then $v > 0$ for all $v \in \mathcal{V}(\bar{\lambda})$. Moreover, $d_K \neq 0$ for all $K \in \mathcal{B}(\bar{\lambda})$ when $p(\cdot) = \|\cdot\|_1$.
\end{proposition}

\begin{proof}
Recall that $\mathcal{B}(\bar{\lambda})$ is nonempty. Let $K \in \mathcal{B}(\bar{\lambda})$ be arbitrarily chosen. We know that $\widehat{A}_{:K}$ is of full column rank. Denote $h^K = \widehat{A}_{:K}(\widehat{A}_{:K}^T\widehat{A}_{:K})^{-1}d_K \in \mathcal{H}(\bar{\lambda})$. Since $d_K \neq 0$, it holds that
\[
\|h_K\|^2 = \langle \widehat{A}_{:K}(\widehat{A}_{:K}^T\widehat{A}_{:K})^{-1}d_K, \widehat{A}_{:K}(\widehat{A}_{:K}^T\widehat{A}_{:K})^{-1}d_K \rangle = \langle d_K, (\widehat{A}_{:K}^T\widehat{A}_{:K})^{-1}d_K \rangle > 0.
\]
Therefore, it follows from \cref{lemma: upper-semicontinuouity} and the facts $\bar{\lambda} > 0$ and $\langle P^K b, h^K \rangle = 0$ that
\[
\varphi(\bar{\lambda}) = \sqrt{\|P^K b\|^2 + \bar{\lambda}^2\|h^K\|^2} > 0,
\]
which implies that
\[
v^K = \bar{\lambda}\|h^K\|^2/\varphi(\bar{\lambda}) \in \mathcal{V}(\bar{\lambda}) \quad \mbox{ and } \quad v^K > 0.
\]
Since $K \in \mathcal{B}(\bar{\lambda})$ is arbitrarily chosen, we know that $v > 0$ for all $v \in \mathcal{V}(\bar{\lambda})$.

When $p(\cdot) = \|\cdot\|_1$, the set $\partial p(0)$ has the representation of
\[
\partial p(0) = \{u \in \mathbb{R}^n \,|\, -1 \leq u_i \leq 1, ~ i \in [n]\}.
\]
In other words, $B = [I_n ~ -I_n]^T \in \mathbb{R}^{2n \times n}$ and $d = e_{2n}$. Therefore, $d_K \neq 0$ for any $\bar{\lambda} \in (0, \lambda_{\infty})$ and $K \in \mathcal{B}(\bar{\lambda})$.
\end{proof}

The numerical results in \cref{sec: numerical-experiment} will show that the secant method and the semismooth Newton method are comparable for solving the least-squares constrained Lasso problem, which also demonstrates the high efficiency of the secant method {even for the case that the HS-Jacobian can be computed}.

\section{An adaptive sieving based secant method for \cref{eq: main-prob}}
\label{sec: AS-levelset}
A main computational challenge  (especially for high dimensional problems) for solving \cref{eq: root-finding-levelset-lst} comes from computing the function value of $\varphi(\cdot)$, which requires solving the optimization problem \cref{eq: reg-LS-form}. To address this challenge, we will incorporate a dimension reduction technique called adaptive sieving to  \cref{alg:hybrid-levelset}.

\subsection{An adaptive sieving technique for sparse optimization problems} We briefly introduce the AS technique developed in \cite{AS-YLST23} for solving sparse optimization problems of the following form:
\begin{equation}
\label{eq: general-sparse-problem}
\min_{x \in \mathbb{R}^n} \left\{
\Phi(x) + P(x)
\right\},
\end{equation}
where $\Phi:\mathbb{R}^n\rightarrow \mathbb{R}$ is a continuously differentiable convex function, and $P: \mathbb{R}^n \to (-\infty,+\infty]$ is a closed proper convex function. We assume that the convex composite optimization problem \cref{eq: general-sparse-problem} has at least one solution. Certainly, the optimization problem \cref{eq: reg-LS-form} is a special case of \cref{eq: general-sparse-problem}. We define the proximal residual function $R: \mathbb{R}^n \to \mathbb{R}^n$ as
\begin{equation}
R(x) := x - {\rm Prox}_{P}(x - \nabla\Phi(x)), \quad x \in \mathbb{R}^n.
\end{equation}
The norm of $R(x)$ is a standard measurement for the quality of an obtained solution, and $x$ is a solution to \cref{eq: general-sparse-problem} if and only if $R(x) = 0$.

Let $I \subseteq [n]$ be an index set. We consider the following constrained optimization problem with the index set $I$:
\begin{equation}
\label{eq: AS-general-sub}
\min_{x \in \mathbb{R}^n} ~\left\{ \Phi(x) + P(x) \ | \  x_{I^c} = 0
\right\},
\end{equation}
where $I^c = [n] \backslash I$ is the complement of $I$. A key fact is that, a solution to \cref{eq: AS-general-sub} is also a solution to \cref{eq: general-sparse-problem} if there exists a solution $\bar{x}$ to \cref{eq: general-sparse-problem} such that $supp(\bar{x}) \subseteq I$. The AS technique is motivated by this fact. Specifically, starting with a reasonable guessing $I_0 \subseteq [n]$, the AS technique is an adaptive strategy to refine the current index set $I_k$ based on a solution to \cref{eq: AS-general-sub} with $I = I_k$. We present the details of the AS technique for solving \cref{eq: general-sparse-problem} in \cref{alg:screening}.

\begin{algorithm}
\caption{An adaptive sieving strategy for solving \eqref{eq: general-sparse-problem}}
	\label{alg:screening}
	\begin{algorithmic}[1]
		\STATE \textbf{Input}: an initial index set $I_0 \subseteq [n]$, a given tolerance $\epsilon\geq 0$ and a given positive integer $k_{\max}$ (e.g., $k_{\max}=500$).
		\STATE \textbf{Output}: a solution $x^*$ to the problem \eqref{eq: general-sparse-problem} satisfying $\|R(x^*)\|\leq \epsilon$.
		\STATE \textbf{1}. Find
		\begin{align}
		x^{0} \in \underset{ x\in \mathbb{R}^n} {\arg\min} \  \Big\{\Phi(x)  +P(x) - \langle \delta^0,x\rangle	\ \mid \ x_{I_0^c}=0\Big\},\label{eq: lambda0_problem}
		\end{align}
		where $\delta^0\in \mathbb{R}^n$ is an error vector such that $\|\delta^0\|\leq \epsilon$ and $(\delta^0)_{I_0^c}=0$.
		
		\textbf{2}. Compute $R(x^0)$ and set $s=0$.
		\WHILE{$\|R(x^{s})\|> \epsilon$}
		\STATE \textbf{3.1}. Create $J_{s+1}$ as
		\begin{align}
		J_{s+1} = \Big\{ j\in I_s^c\; \mid \; (R(x^s))_j\neq 0\Big\}.\label{eq: create_J}
		\end{align}
        If $J_{s+1} = \emptyset$, let $I_{s+1} \leftarrow I_{s}$; otherwise, let $k$ be a positive integer satisfying $k \leq \min \{ |J_{s+1}|, k_{\max}\}$ and define
         \[\scriptsize
         \widehat{J}_{s+1} = \Big\{j \in J_{s+1} \,|\, \mbox{ $|(R(x^s))_j|$ is among the first $k$ largest values in $\{ |(R(x^s))_i| \}_{i \in J_{s+1} } $ } \Big\}.
        \]
        Update $I_{s+1}$ as:		
		\begin{align*}
		I_{s+1} \leftarrow I_{s} \cup \widehat{J}_{s+1}.
		\end{align*}
		\STATE \textbf{3.2}. Solve the constrained problem:
		\begin{align}
		x^{s+1} \in \underset{ x\in \mathbb{R}^n} {\arg\min} \  \Big\{\Phi(x)  + P(x) - \langle \delta^{s+1},x\rangle \ \mid \ x_{I_{s+1}^c}=0\Big\},\label{eq: constrained}
		\end{align}
		where $\delta^{s+1}\in \mathbb{R}^n$ is an error vector such that $\|\delta^{s+1}\|\leq \epsilon$ and $(\delta^{s+1})_{I_{s+1}^c}=0$.
		\STATE \textbf{3.3}: Compute $R(x^{s+1})$ and set $s\leftarrow s+1$.
		\ENDWHILE
		\STATE \textbf{return}: Set $x^*=x^{s}$.
	\end{algorithmic}
\end{algorithm}

It is worthwhile mentioning that, in \cref{alg:screening}, the error vectors $\delta^0$, $\{\delta^{s+1}\}$ in \eqref{eq: lambda0_problem} and \eqref{eq: constrained} are not given but imply that the corresponding minimization problems can be solved inexactly. We can just take $\delta^s = 0$ (for $s \geq 0$) if we solve the reduced subproblems exactly. The following proposition shows that we can obtain an inexact solution by solving a reduced problem with a much smaller dimension.
\begin{proposition}[\mbox{\cite[Proposition 1]{AS-YLST23}}]
\label{lemma: YLST-21-prop1}
For any given nonnegative integer $s$, the updating rule of $x^s$ in Algorithm \ref{alg:screening} can be interpreted in the procedure as follows. Let $M_s$ be a linear map from $\mathbb{R}^{|I_s|}$ to $\mathbb{R}^n$ defined as
	\begin{align*}
		(M_s z)_{I_s} = z,\quad (M_s z)_{I_s^c} = 0,\quad z\in \mathbb{R}^{|I_s|},
	\end{align*}
    and $\Phi^s$, $P^s$ be functions from $\mathbb{R}^{|I_s|}$ to $\mathbb{R}$ defined as $\Phi^s(z):= \Phi(M_s z)$, $P^s(z) := P(M_s z)$ for all $z\in \mathbb{R}^{|I_s|}$. Then $x^s\in \mathbb{R}^n$ can be computed as
	\begin{align*}
	(x^s)_{I_s}:={\rm Prox}_{P^s}(\hat{z}-\nabla \Phi^s(\hat{z})),
	\end{align*}
	and $(x^s)_{I_s^c}=0$, where $\hat{z}$ is an approximate solution to the problem
	\begin{align}
		\min_{z\in \mathbb{R}^{|I_s|}}\ \Big\{ \Phi^s(z) + P^s(z)\Big\},\label{eq: reduced_0}
	\end{align}
    which satisfies
	\begin{align}
	\|\hat{z}-{\rm Prox}_{P^s}(\hat{z}-\nabla \Phi^s(\hat{z}))+\nabla \Phi^s({\rm Prox}_{P^s}(\hat{z}-\nabla \Phi^s(\hat{z})))-\nabla \Phi^s(\hat{z})\|\leq \epsilon,\label{eq: control_delta0}
	\end{align}
	and $\epsilon$ is the parameter given in Algorithm \ref{alg:screening}.
\end{proposition}
The finite termination property of \cref{alg:screening} for solving \cref{eq: general-sparse-problem} is shown in the following proposition.
\begin{proposition}
[\mbox{\cite[Theorem 1]{AS-YLST23}}]
\label{finite-termination-AS}
For any given initial index set $I_0 \subseteq [n]$ and tolerance $\epsilon \geq 0$, the while loop in Algorithm \ref{alg:screening} will terminate after a finite number of iterations.
\end{proposition}

The high efficiency of the AS technique for solving a wide class of sparse optimization problems in the form of \cref{eq: general-sparse-problem} has been demonstrated in \cite{AS-YLST23,AS-YST22,LQ-AS-23,WCLS-23}.

\subsection{An adaptive sieving based secant method for \cref{eq: main-prob}} When applying the level-set method to solve \cref{eq: main-prob}, one needs to solve a sequence of regularized problems in the form of \cref{eq: reg-LS-form} with $\lambda \in \{\lambda_k\}_{k \geq 0}$ generated by the root-finding algorithm for solving \cref{eq: root-finding-levelset-lst}. This naturally motivates us to incorporate \cref{alg:screening} into \cref{alg:hybrid-levelset} for solving the regularized least-squares problems in the form of \cref{eq: reg-LS-form} since we can effectively construct an initial index set for the AS technique based on the solution of the previous problem on the solution path. For the first problem on the solution path, we will choose $\lambda_0$ to be relatively large such that its solution is highly sparse. In such a way, we choose $I_0 = \emptyset$ in \cref{alg:screening} for the first problem.  We present the details in \cref{alg: AS-secant}.

\begin{algorithm}
	\caption{SMOP: A root finding based secant method for \cref{eq: main-prob}}
	\label{alg: AS-secant}
	\begin{algorithmic}[1]
		\STATE{\textbf{Input}: $A \in \mathbb{R}^{m \times n}$, $b \in \mathbb{R}^n$, $\mu \in (0,1)$, $0< \underline{\lambda} < \lambda_{1} < \lambda_0 \leq \overline{\lambda} \leq \lambda_{\infty}$ satisfying $\varphi(\underline{\lambda}) < \varrho < \varphi(\overline{\lambda})$. }
		\STATE{
        \textbf{Step 1}. Call \cref{alg:screening} with $I_0 = \emptyset$ to solve \cref{eq: reg-LS-form} with $\lambda = \lambda_0$ and obtain the solution $x(\lambda_0)$. Compute $\varphi(\lambda_0)$.}
        \WHILE{$\varphi(\lambda_k) \neq \varrho$}
        \STATE{\textbf{Step 2.1}. Set $k = k+1$. Set $I^k_0 = \{i \in [n] \,|\, (x(\lambda_{k-1}))_i \neq 0\}$.}
        \STATE{
        \textbf{Step 2.2}. Call \cref{alg:screening} with $I_0 = I_0^k$ to solve \cref{eq: reg-LS-form} with $\lambda = \lambda_k$ to obtain $x(\lambda_k)$ and compute $\varphi(\lambda_k)$.}
        \STATE{\textbf{Step 2.3}. Generate $\lambda_{k+1}$ by \cref{alg:hybrid-levelset}.}
        \ENDWHILE
        \STATE{\textbf{Step 3}. Set $\lambda^* = \lambda_k$ and $x^* = x(\lambda_k)$.}
        \STATE{\textbf{Return}: $\lambda^*$ and $x^*$.}
	\end{algorithmic}
\end{algorithm}

We make some remarks before concluding this section. Firstly, we can naturally apply \cref{alg: AS-secant} to efficiently generate a solution path for \cref{eq: main-prob} with a sequence of noise-level controlling parameters $\varrho_1 > \varrho_2 > \cdots > \varrho_T > 0$. Secondly, we know that, if we apply the AS technique to the level-set method based on \cref{eq: root-finding-origin-levelset}, we may easily encounter the infeasibility issue if $\tau_{k+1} < \tau_k$ in \cref{eq: origin-levelset-subprob}.

\section{Numerical experiments}
\label{sec: numerical-experiment}
In this section, we will present numerical results to demonstrate the high efficiency of our proposed SMOP. We will focus on solving \cref{eq: main-prob} with two objective functions: (1) the $\ell_1$ penalty:  $p(x) = \norm{x}_1$, $x \in \mathbb{R}^n$; (2) the sorted $\ell_1$ penalty:  $p(x) = \sum_{i = 1}^n \gamma_i |x|_{(i)}$, $x \in \mathbb{R}^n$ with given parameters $\gamma_1 \geq \gamma_2 \geq \cdots \geq \gamma_n \geq 0$ and $\gamma_1 > 0$, where $|x|_{(1)} \geq |x|_{(2)} \geq \cdots \geq |x|_{(n)}$, which serve as illustrative examples to highlight the efficiency of our algorithm. It is worthwhile mentioning that the sorted $\ell_1$ penalty is not separable. For demonstration purposes only, we will test the performance of SMOP when $p(\cdot)$ is a non-polyhedral function at the end of this section.
\begin{table}[htbp]
	\centering
	\tiny
	\caption{Statistics of the UCI test instances.}
	\begin{tabular}{cccccc}
		\toprule
		Problem idx & Name  & m     & n     & Sparsity(A) & norm(b) \\
		\midrule
		1     & E2006.train & 16087 & 150360 & 0.0083 & 452.8605 \\
		2     & log1p.E2006.train & 16087 & 4272227 & 0.0014 & 452.8605 \\
		3     & E2006.test & 3308  & 150358 & 0.0092 & 221.8758 \\
		4     & log1p.E2006.test & 3308  & 4272226 & 0.0016 & 221.8758 \\
		5     & pyrim5 & 74    & 201376 & 0.5405 & 5.7768 \\
		6     & triazines4 & 186   & 635376 & 0.6569 & 9.1455 \\
		7     & bodyfat7 & 252   & 116280 & 1.0000 & 16.7594 \\
		8     & housing7 & 506   & 77520 & 1.0000 & 547.3813 \\
		\bottomrule
	\end{tabular}%
	\label{tab:ucidetails}
\end{table}%

In our numerical experiments, we measure the accuracy of the obtained solution $\tilde{x}$ for \cref{eq: main-prob} by the following relative residual:
\[
	\eta := \frac{|\tilde\varphi - \varrho|}{\max\{1,\,\varrho\}},
\]
where $\tilde{\varphi} := \norm{A\tilde{x} - b}$.
We test all algorithms on datasets from UCI Machine Learning Repository as in \cite{li2018highly,li2018efficiently}, which are originally obtained from the LIBSVM datasets \cite{chih2011libsvm}. Table \ref{tab:ucidetails} presents the statistics of the tested UCI instances. All our computational results are obtained using MATLAB R2023a on a Windows workstation with the following specifications: 12-core Intel(R) Core(TM) i7-12700 (2.10GHz) processor, and 64 GB of RAM. {In all the tables presented in this section, $\text{nnz}(x)$ represents the number of elements in the solution $x$ obtained by SMOP (with a stopping tolerance of $10^{-6}$) for solving \cref{eq: main-prob} that have an absolute value greater than $10^{-8}$. Besides, We denote BMOP (NMOP) as the root finding based bisection method (hybrid of the bisection method and the semismooth Newton method) for solving the optimization problem \cref{eq: main-prob}.
}

\subsection{The $\ell_1$ penalized problems with least-squares constraints}
In this subsection, we focus on the problem \cref{eq: main-prob} with $p(\cdot) = \norm{\cdot}_1$. We will compare the efficiency of SMOP to the state-of-the-art SSNAL-LSM algorithm \cite{li2018efficiently}, SPGL1 solver \cite{Van08levelset, spgl1site} and ADMM.
Moreover, we perform experiments to demonstrate that our secant method is considerably more efficient than the bisection method for root finding while performing on par with the semismooth Newton method, where the HS-Jacobian is computable.

\begin{table}[htbp]
	\centering
	\tiny
	\caption{The values of $c$ to obtain $\varrho = c \norm{b}$ {for the $\ell_1$ penalized problems with least-squares constraints}. In this table, $c_{LS} = \frac{\lambda^*}{\norm{A^Tb}_\infty}$ represents the regularization parameter for the corresponding ${\rm P_{LS}}(\lambda^*)$, where the optimal solution $\lambda^*$ to $\varphi(\lambda) = \varrho$ is obtained by SMOP.}
	\begin{tabular}{cccccccccc}
		\hline
		Test   & idx   & 1     & 2     & 3     & 4     & 5     & 6         & 7     & 8     \bigstrut\\
		\hline
		\multirow{3}[2]{*}{I} & c     & 0.1   & 0.1   & 0.08  & 0.08  & 0.05  & 0.1    & 0.001 & 0.1    \bigstrut[t]\\
		& nnz(x) & 339   & 110   & 246   & 405   & 79    & 655      & 107   & 148    \\
		&$c_{LS}$ & 2.6-7 & 2.8-4 & 4.2-7 & 2.1-4 & 5.7-3 & 2.8-3  & 1.1-6 & 1.3-3 \bigstrut[b]\\
		\hline
		\multirow{3}[2]{*}{II} & c     & 0.09  & 0.09  & 0.06  & 0.06  & 0.015 & 0.03   & 0.0001 & 0.04   \bigstrut[t]\\
		& nnz(x) & 1387  & 1475  & 884   & 1196  & 92    & 497      & 231   & 377    \\
		& $c_{LS}$ & 1.1-7 & 6.2-5 & 1.7-7 & 9.6-5 & 3.0-4 & 5.6-5  & 3.8-8 & 3.0-5  \bigstrut[b]\\
		\hline
	\end{tabular}%
	\label{tab:testc}%
\end{table}%

\begin{figure}
	\centering
	\begin{minipage}[b]{0.7\linewidth}
		\centering
		\includegraphics[width=\textwidth]{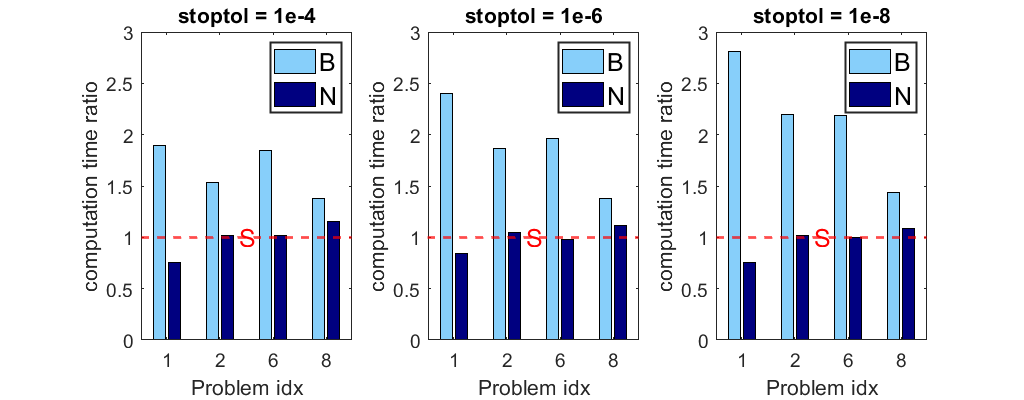}
		\caption*{Test I}
	\end{minipage}
	\\
	\begin{minipage}[b]{0.7\linewidth}
		\centering
		\includegraphics[width=\textwidth]{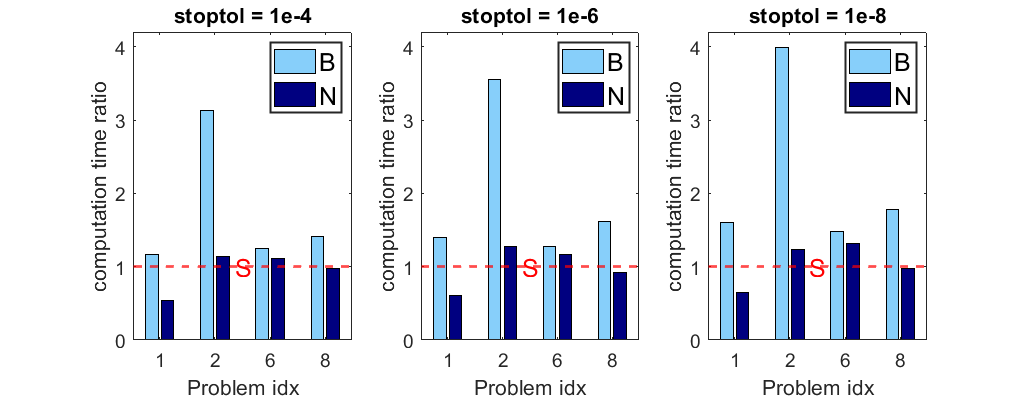}
		\caption*{Test II}
	\end{minipage}
 	\caption{The ratio of the computation time between BMOP (B) and NMOP (N) to the computation time of SMOP in solving \cref{eq: main-prob} { with the $\ell_1$ regularization}.}
	\label{fig:ratioBSN}
 \vspace{-0.6cm}
\end{figure}

In practice, we have multiple choices for solving the sub-problems in SMOP. In our experiments, we utilized the squared smoothing Newton method \cite{qi2000new, gao2009calibrating} and SSNAL to solve the sub-problems in SMOP. The maximum number of iterations for SPGL1, SSNAL, and ADMM is set to 100,000, while for SMOP and SSNAL-LSM, the maximum number of iterations of the outermost loop is set to 200. {Additionally, we have set the maximum running time to 1 hour.} To select $\varrho$, we use the values of $c$ in \cref{tab:testc} for each instance listed in \cref{tab:ucidetails}, and let $\varrho = c\norm{b}$. {At last, we point out that the adaptive sieving technique is not employed in the SSNAL-LSM.}

\begin{table}[htbp]
  \centering
  \tiny
  \caption{The performance of SMOP (A1), SSNAL-LSM (A2), SPGL1 (A3) and ADMM (A4), in solving {the $\ell_1$ penalized problems with least-squares constraints} \cref{eq: main-prob} with $\varrho = c\norm{b}$, where the specific value of $c$ for each problem is listed in \cref{tab:testc}. In the table, the underline is used to mark cases where the algorithm fails to reach the given tolerance. For simplicity, we omit the ``e" in the scientific notation.}
    \begin{tabular}{|c|c|c|c|}
    \hline
    \multirow{2}[4]{*}{\textbf{idx}} & \textbf{time (s)} & \boldmath{}\textbf{$\eta$}\unboldmath{} & \textbf{outermost iter} \bigstrut\\
\cline{2-4}          & \textbf{A1 $|$ A2 $|$ A3 $|$ A4} & \textbf{A1 $|$ A2 $|$ A3 $|$ A4} & \textbf{A1 $|$ A2 $|$ A3 $|$ A4} \bigstrut\\
    \hline
          & \multicolumn{3}{c|}{Test I with stoptol = $10^{-4}$} \bigstrut[t]\\
    \hline
    \textbf{1} & 1.39+0 $|$ 2.18+2 $|$ 3.51+2 $|$ 4.22+2 & 2.3-5 $|$ 4.9-5 $|$ 1.0-4 $|$ 1.0-4 & 24 $|$ 29 $|$ 7342 $|$ 2049 \bigstrut[t]\\
    \textbf{2} & 2.29+0 $|$ 5.12+2 $|$ 1.45+3 $|$ 6.84+2 & 3.1-6 $|$ 7.8-5 $|$ 9.0-5 $|$ 8.7-5 & 12 $|$ 16 $|$ 3445 $|$ 1470 \\
    \textbf{3} & 4.02$-$1 $|$ 5.83+1 $|$ 3.21+2 $|$ 8.87+1 & 9.4-6 $|$ 2.6-5 $|$ 1.0-4 $|$ 1.0-4 & 24 $|$ 30 $|$ 21094 $|$ 4918 \\
    \textbf{4} & 1.59+0 $|$ 2.06+2 $|$ 7.19+2 $|$ 9.90+1 & 1.2-5 $|$ 7.3-5 $|$ 9.5-5 $|$ 1.3-5 & 13 $|$ 15 $|$ 3174 $|$ 854 \\
    \textbf{5} & 2.73$-$1 $|$ 1.20+1 $|$ 9.81+0 $|$ 5.63+0 & 6.9-6 $|$ 5.4-6 $|$ 7.4-5 $|$ 2.2-5 & 6 $|$ 14 $|$ 498 $|$ 273 \\
    \textbf{6} & 2.32+0 $|$ 1.74+2 $|$ 3.35+2 $|$ 1.01+2 & 5.8-6 $|$ 4.4-5 $|$ 9.1-5 $|$ 7.5-5 & 9 $|$ 17 $|$ 1987 $|$ 571 \\
    \textbf{7} & 4.35$-$1 $|$ 9.12+0 $|$ 8.98+0 $|$ 8.59+0 & 2.8-5 $|$ 5.9-5 $|$ 9.8-5 $|$ 9.9-5 & 15 $|$ 18 $|$ 539 $|$ 583 \\
    \textbf{8} & 2.99$-$1 $|$ 9.07+0 $|$ 1.29+1 $|$ 7.94+0 & 2.6-5 $|$ 8.6-5 $|$ 1.0-4 $|$ 9.0-5 & 10 $|$ 14 $|$ 515 $|$ 424 \\
    \hline
          & \multicolumn{3}{c|}{Test I with stoptol = $10^{-6}$} \bigstrut[t]\\
    \hline
    \textbf{1} & 1.45+0 $|$ 3.22+2 $|$ 1.51+3 $|$ 7.06+2 & 2.5-7 $|$ 6.1-8 $|$ 9.9-7 $|$ 1.0-6 & 25 $|$ 36 $|$ 28172 $|$ 3539 \bigstrut[t]\\
    \textbf{2} & 2.52+0 $|$ 6.68+2 $|$ 1.75+3 $|$ 3.42+3 & 9.9-8 $|$ 3.5-8 $|$ 9.2-7 $|$ 9.9-7 & 13 $|$ 24 $|$ 4155 $|$ 8725 \\
    \textbf{3} & 4.12$-$1 $|$ 7.40+1 $|$ \underline{2.11+3} $|$ 1.81+2 & 1.1-8 $|$ 2.3-7 $|$ \underline{6.2-6} $|$ 1.0-6 & 25 $|$ 35 $|$ \underline{100000} $|$ 10100 \\
    \textbf{4} & 1.72+0 $|$ 3.40+2 $|$ 1.04+3 $|$ 4.03+2 & 1.3-9 $|$ 5.7-7 $|$ 7.2-7 $|$ 7.9-7 & 14 $|$ 26 $|$ 4584 $|$ 3820 \\
    \textbf{5} & 2.93$-$1 $|$ 1.61+1 $|$ 4.58+1 $|$ 3.95+2 & 1.0-7 $|$ 6.0-8 $|$ 9.1-7 $|$ 9.8-7 & 7 $|$ 19 $|$ 2468 $|$ 20155 \\
    \textbf{6} & 2.47+0 $|$ 2.13+2 $|$ 8.24+2 $|$ 2.31+3 & 3.0-7 $|$ 4.0-7 $|$ 8.2-7 $|$ 3.4-7 & 10 $|$ 23 $|$ 5578 $|$ 13672 \\
    \textbf{7} & 4.68$-$1 $|$ 1.18+1 $|$ 9.11+0 $|$ 1.85+1 & 1.9-9 $|$ 9.6-7 $|$ 2.7-7 $|$ 9.9-7 & 17 $|$ 22 $|$ 544 $|$ 1250 \\
    \textbf{8} & 3.28$-$1 $|$ 1.45+1 $|$ 3.84+1 $|$ 4.40+1 & 2.4-7 $|$ 8.4-8 $|$ 4.0-7 $|$ 8.7-7 & 11 $|$ 24 $|$ 1539 $|$ 2427 \\
    \hline
          & \multicolumn{3}{c|}{Test II with stoptol = $10^{-4}$} \bigstrut[t]\\
    \hline
    \textbf{1} & 7.26+0 $|$ 4.51+2 $|$ 1.38+3 $|$ 6.12+2 & 3.0-6 $|$ 4.6-5 $|$ 1.0-4 $|$ 1.0-4 & 26 $|$ 30 $|$ 27775 $|$ 3014 \bigstrut[t]\\
    \textbf{2} & 6.79+0 $|$ 1.54+3 $|$ 1.32+3 $|$ 4.01+2 & 1.8-5 $|$ 3.6-5 $|$ 9.7-5 $|$ 6.8-5 & 14 $|$ 21 $|$ 3000 $|$ 733 \\
    \textbf{3} & 3.51+0 $|$ 1.84+2 $|$ \underline{1.50+3} $|$ 1.34+2 & 1.3-5 $|$ 2.3-5 $|$ \underline{8.7-2} $|$ 1.0-4 & 25 $|$ 29 $|$ \underline{100000} $|$ 7333 \\
    \textbf{4} & 2.91+0 $|$ 6.91+2 $|$ 6.23+2 $|$ 4.94+1 & 7.5-6 $|$ 3.6-6 $|$ 9.6-5 $|$ 5.8-5 & 14 $|$ 22 $|$ 2694 $|$ 385 \\
    \textbf{5} & 6.23$-$1 $|$ 1.53+1 $|$ 8.65+0 $|$ 2.01+1 & 2.8-5 $|$ 7.9-6 $|$ 6.6-5 $|$ 9.5-5 & 9 $|$ 13 $|$ 395 $|$ 1000 \\
    \textbf{6} & 9.02+0 $|$ 3.46+2 $|$ \underline{3.60+3} $|$ 3.82+2 & 6.8-6 $|$ 3.7-5 $|$ \underline{7.6-2} $|$ 9.9-5 & 12 $|$ 17 $|$ \underline{24924} $|$ 2232 \\
    \textbf{7} & 1.50+0 $|$ 1.59+1 $|$ 3.06+2 $|$ 3.39+1 & 1.6-5 $|$ 8.7-6 $|$ 9.9-5 $|$ 9.8-5 & 12 $|$ 18 $|$ 19820 $|$ 2340 \\
    \textbf{8} & 2.37+0 $|$ 1.90+1 $|$ 1.69+2 $|$ 1.19+1 & 1.4-6 $|$ 8.9-5 $|$ 9.1-5 $|$ 9.8-5 & 13 $|$ 18 $|$ 5914 $|$ 644 \\
    \hline
          & \multicolumn{3}{c|}{Test II with stoptol = $10^{-6}$} \bigstrut[t]\\
    \hline
    \textbf{1} & 7.23+0 $|$ 5.96+2 $|$ \underline{3.60+3} $|$ 8.82+2 & 3.7$-$9 $|$ 2.9-7 $|$ \underline{3.6-2} $|$ 1.0-6 & 27 $|$ 35 $|$ \underline{62384} $|$ 4453 \bigstrut[t]\\
    \textbf{2} & 7.37+0 $|$ 1.85+3 $|$ 2.04+3 $|$ 1.46+3 & 1.4$-$7 $|$ 3.9-7 $|$ 9.7-7 $|$ 1.0-7 & 15 $|$ 27 $|$ 4688 $|$ 3464 \\
    \textbf{3} & 3.59+0 $|$ 2.36+2 $|$ \underline{1.49+3} $|$ 1.99+2 & 8.1-10 $|$ 8.3-7 $|$ \underline{8.7-2} $|$ 1.0-6 & 26 $|$ 36 $|$ \underline{100000} $|$ 11051 \\
    \textbf{4} & 3.02+0 $|$ 8.44+2 $|$ 1.37+3 $|$ 2.18+2 & 3.1$-$9 $|$ 4.3-7 $|$ 9.9-7 $|$ 6.0-7 & 15 $|$ 28 $|$ 5912 $|$ 1980 \\
    \textbf{5} & 6.37$-$1 $|$ 2.49+1 $|$ 4.14+2 $|$ 1.48+2 & 2.4$-$7 $|$ 3.0-8 $|$ 8.7-7 $|$ 9.7-7 & 10 $|$ 22 $|$ 22091 $|$ 7592 \\
    \textbf{6} & 9.37+0 $|$ 4.25+2 $|$ \underline{3.60+3} $|$ 3.60+3 & 5.4-11 $|$ 6.7-7 $|$ \underline{7.5-2} $|$ 1.2-7 & 14 $|$ 22 $|$ \underline{25158} $|$ 21556 \\
    \textbf{7} & 1.59+0 $|$ 2.09+1 $|$ 3.37+2 $|$ 8.54+1 & 3.2$-$7 $|$ 1.9-8 $|$ 8.8-7 $|$ 9.7-7 & 13 $|$ 23 $|$ 21523 $|$ 5817 \\
    \textbf{8} & 2.39+0 $|$ 2.68+1 $|$ 1.65+3 $|$ 3.34+1 & 4.5$-$7 $|$ 6.9-7 $|$ 8.8-7 $|$ 9.8-7 & 14 $|$ 26 $|$ 59147 $|$ 1834 \\
    \hline
    \end{tabular}%
  \label{tab:lassoCP}%
\end{table}%

We compare SMOP to SPGL1, SSNAL-LSM and ADMM to solve \cref{eq: main-prob} with the tolerances of $10^{-4}$ and $10^{-6}$, respectively. The test results are presented in \cref{tab:lassoCP}. These results indicate that SMOP successfully solves all the tested instances and outperforms SSNAL-LSM, SPGL1 and ADMM. It can be seen from \cref{tab:lassoCP} that SMOP can achieve a speed-up of up to 1,000 times compared to SPGL1 for the problems that can be solved by SPGL1 (a significant number of instances cannot be solved by SPGL1 to the required accuracy). {Regarding ADMM, SMOP remains significantly superior in terms of efficiency for all cases, with a speed-up of over 1300 times. } Furthermore, compared to SSNAL-LSM, SMOP also shows {vast} superiority, with efficiency improvements up to more than 260 times. Note that SPGL1 has two modes: the primal mode (denoted by SPGL1) and the hybrid mode (denoted by SPGL1\_H). We do not print the results of SPGL1\_H since SPGL1 outperforms SPGL1\_H in most of the cases in our tests.

Subsequently, we perform numerical experiments to compare the performance of the secant method to the bisection method and the HS-Jacobian based semismooth Newton method for finding the root of \cref{eq: root-finding-levelset-lst} to further illustrate the efficiency of SMOP.
\cref{fig:ratioBSN} presents the ratio of the computation time between BMOP and NMOP to the computation time of SMOP on solving \cref{eq: main-prob} for some instances. The numerical results show that SMOP easily beats BMOP, with a large margin when a higher precision solution is required. The results also reveal that SMOP performs comparably to NMOP for solving the $\ell_1$ penalized {least-squares constrained} problem, in which the HS-Jacobian are computable. This is another strong evidence to illustrate the efficiency of SMOP.

\begin{figure}[htbp]
	\centering	
	\begin{subfigure}[b]{0.35\textwidth}
		\includegraphics[width=\textwidth]{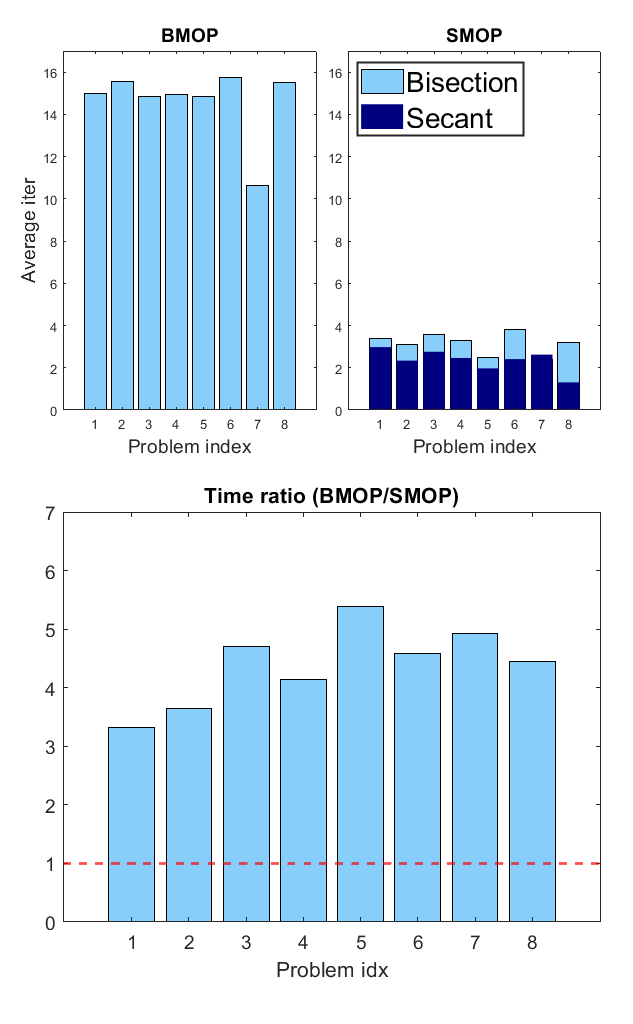}
		\caption*{Test I}
	\end{subfigure}
	\begin{subfigure}[b]{0.35\textwidth}
		\includegraphics[width=\textwidth]{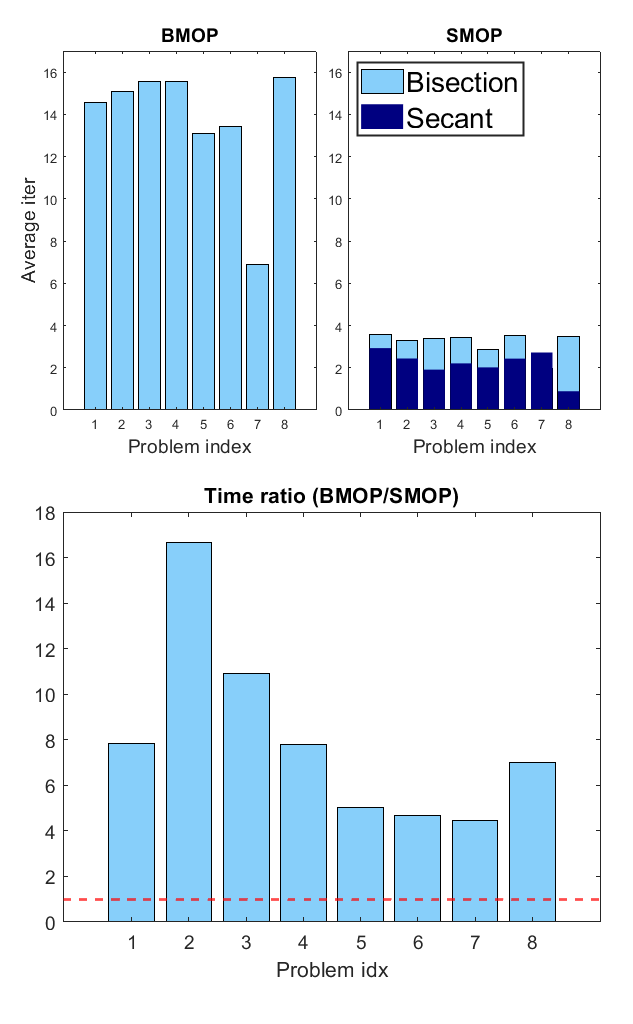}
		\caption*{Test II}
	\end{subfigure}
	\caption{The performance of BMOP and SMOP in generating a solution path for \cref{eq: main-prob} { with the $\ell_1$ regularization} and stopping tolerance $10^{-6}$.}
	\label{fig:solutionpath}
 \vspace{-0.6cm}
\end{figure}

Next we perform tests on BMOP and SMOP to generate a solution path for \cref{eq: main-prob} involving multiple choices of tuning parameters $\varrho > 0$. In this test, we solve \cref{eq: main-prob} with $\varrho_i = c_i \cdot c \norm{b}, \, i = 1, \dots, 100$, where $c_i = 1.5 - 0.5 \times (i-1)/99$ and $c$ is the same constant as in Table \ref{tab:testc}.
In this test, we apply the warm-start strategy to both algorithms. The average iteration numbers of BMOP and SMOP and the ratio of the computation time of BMOP to the computation time of SMOP are shown in \cref{fig:solutionpath} with a tolerance of $10^{-6}$.
From this figure, it is evident that utilizing the secant method for root-finding significantly reduces the number of iterations by around 4 times. {The reduction in iterations results in a substantial decrease in computation time for SMOP, which is typically less than one-third of the time required by BMOP.}

\subsection{The sorted $\ell_1$ penalized problems with least-squares constraints}

In this subsection, we will present the numerical results of SMOP in solving the sorted $\ell_1$ penalized problems with least-squares constraints \cref{eq: main-prob}. {For comparison purposes, we also conducted tests on Newt-ALM-LSM (similar to SSNAL-LSM, but with the sub-problems solved by Newt-ALM \cite{luo2019solving}) and ADMM for \cref{eq: main-prob}.}

\begin{table}[htbp]
  \centering
  \tiny
  \caption{The performance of SMOP (A1), Newt-ALM-LSM (A2) and ADMM (A4), in solving {the sorted $\ell_1$ penalized problems with least-squares constraints} \cref{eq: main-prob} with $\varrho = c\norm{b}$. In the table, $c_{LS} = \frac{\lambda^*}{\norm{A^Tb}_\infty}$ represents the regularization parameter for the corresponding ${\rm P_{LS}}(\lambda^*)$, where the optimal solution $\lambda^*$ to $\varphi(\lambda) = \varrho$ is obtained by SMOP. The stopping tolerance is set to $10^{-6}$ and the underline is used to mark cases where the algorithm fails to reach the given tolerance. For simplicity, we omit the ``e" in the scientific notation.}
    \begin{tabular}{|c|c|c|c|c|}
    \hline
    \multirow{2}[4]{*}{idx} & \multirow{2}[4]{*}{c $|$ nnz(x) $|$ $c_{LS}$} & time (s) & $\eta$ & outermost iter \bigstrut\\
\cline{3-5}          &       & A1 $|$ A2 $|$ A4 & A1 $|$ A2 $|$ A4 & A1 $|$ A2 $|$ A4 \bigstrut\\
    \hline
    \multicolumn{5}{|c|}{Test I} \bigstrut\\
    \hline
    2     & 0.15 $|$ 3 $|$ 2.4-2 & 3.84+0 $|$ 1.34+2 $|$ \underline{3.60+3} & 1.1-7 $|$ 5.3-7 $|$ \underline{2.8-1} & 8 $|$ 21 $|$ \underline{8637} \bigstrut[t]\\
    4     & 0.1 $|$ 3 $|$ 4.8-3 & 4.79+0 $|$ 1.35+2 $|$ \underline{3.60+3} & 6.0-7 $|$ 8.9-7 $|$ \underline{2.9-4} & 10 $|$ 17 $|$ \underline{28891} \\
    5     & 0.1 $|$ 113 $|$ 1.9-2 & 6.29$-$1 $|$ 4.98+1 $|$ 4.23+2 & 1.0-7 $|$ 4.5-7 $|$ 1.5-7 & 7 $|$ 22 $|$ 17974 \\
    6     & 0.15 $|$ 413 $|$ 1.0-2 & 3.10+0 $|$ 2.43+2 $|$ \underline{3.60+3} & 2.7-7 $|$ 1.6-7 $|$ \underline{1.9-4} & 9 $|$ 21 $|$ \underline{19071} \\
    7     & 0.002 $|$ 22 $|$ 1.9-5 & 3.56$-$1 $|$ 1.67+1 $|$ 2.44+1 & 3.6-9 $|$ 6.0-7 $|$ 9.9-7 & 14 $|$ 22 $|$ 1616 \\
    8     & 0.15 $|$ 95 $|$ 6.9-3 & 6.06$-$1 $|$ 2.55+1 $|$ 1.57+2 & 1.3-7 $|$ 7.7-7 $|$ 9.0-7 & 10 $|$ 23 $|$ 8329 \\
    \hline
    \multicolumn{5}{|c|}{Test II} \bigstrut\\
    \hline
    1     & 0.1 $|$ 339 $|$ 2.6-7 & 2.53+1 $|$ 1.40+2 $|$ 5.13+2 & 2.9-7 $|$ 5.6-7 $|$ 1.0-6 & 25 $|$ 34 $|$ 2490 \bigstrut[t]\\
    2     & 0.095 $|$ 629 $|$ 1.0-4 & 5.39+1 $|$ 4.82+2 $|$ 2.87+3 & 1.7-7 $|$ 2.9-7 $|$ 9.4-7 & 17 $|$ 27 $|$ 6770 \\
    3     & 0.08 $|$ 246 $|$ 4.2-7 & 4.98+0 $|$ 6.54+1 $|$ 1.60+2 & 2.0-8 $|$ 7.1-7 $|$ 1.0-6 & 25 $|$ 36 $|$ 8491 \\
    4     & 0.07 $|$ 758 $|$ 1.4-4 & 2.26+1 $|$ 4.26+2 $|$ 5.86+2 & 4.0-8 $|$ 9.0-7 $|$ 9.8-7 & 16 $|$ 27 $|$ 4550 \\
    5     & 0.02 $|$ 95 $|$ 5.7-4 & 2.05+0 $|$ 9.87+1 $|$ 3.58+2 & 3.2-8 $|$ 5.6-7 $|$ 7.6-7 & 11 $|$ 20 $|$ 15582 \\
    6     & 0.05 $|$ 997 $|$ 5.5-4 & 2.32+1 $|$ 1.04+3 $|$ \underline{3.60+3} & 8.4-7 $|$ 2.1-7 $|$ \underline{3.5-6} & 10 $|$ 23 $|$ \underline{19159} \\
    7     & 0.001 $|$ 107 $|$ 1.1-6 & 1.02+0 $|$ 2.85+1 $|$ 1.30+1 & 5.9-8 $|$ 6.9-9 $|$ 9.5-7 & 17 $|$ 22 $|$ 826 \\
    8     & 0.08 $|$ 206 $|$ 4.3-4 & 3.38+0 $|$ 1.03+2 $|$ 5.58+1 & 5.7-9 $|$ 7.4-7 $|$ 3.8-7 & 13 $|$ 25 $|$ 2842 \\
    \hline
    \end{tabular}%
  \label{tab:slopecomp}%
\end{table}%

In our numerical experiments, we choose the parameters $\gamma_i = 1 - (i-1)/(n-1), \, i = 1, \dots, n,$ in the sorted $\ell_1$ penalty function $p(x) = \sum_{i = 1}^n \gamma_i |x|_{(i)}, \ x \in \mathbb{R}^n$.
The maximum iteration number is set to 200 for SMOP and Newt-ALM-LSM, and 100,000 for ADMM. {The sub-problems in SMOP are solved by Newt-ALM in this test. In all experiments within this subsection, the stopping tolerance is set to $10^{-6}$. The results obtained with a tolerance of $10^{-4}$ are similar to those obtained with a tolerance of $10^{-6}$, therefore, we will not present them in order to save space.}
We set $\varrho = c\norm{b}$, where the values of $c$ are specified in \cref{tab:slopecomp}.
The numerical results are presented in \cref{tab:slopecomp}. {From the table, it is evident that SMOP outperforms Newt-ALM-LSM and ADMM for all the cases. More specifically, SMOP can be up to around 80 times faster than Newt-ALM-LSM and up to more than 600 times faster than ADMM for the problems that can be solved by ADMM. Additionally, Figure \ref{fig:ratioBSslope} presents the computation time ratio between BMOP and SMOP for both Test I and Test II. This also demonstrates the significance of the secant method in root-finding for achieving higher efficiency.}

\begin{figure}
	\centering
		\includegraphics[width=0.8\textwidth]{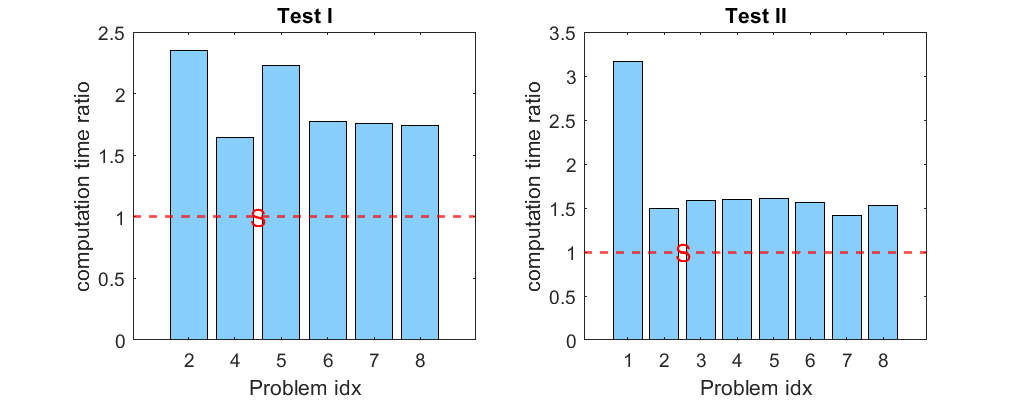}
 	\caption{The ratio of the computation time between BMOP to the computation time of SMOP in solving \cref{eq: main-prob} {with the sorted $\ell_1$ regularization}.}
	\label{fig:ratioBSslope}
\vspace{-0.6cm}
\end{figure}

\subsection{A group lasso penalized problems with least-squares constraints}

\begin{table}[htbp]
  \centering
  \tiny
  \caption{{The values of $c$ to obtain $\varrho = c \norm{b}$ for the group lasso penalized problems with least-squares constraints. In the table, $c_{LS} = \frac{\lambda^*}{\norm{A^Tb}_\infty}$ represents the regularization parameter for the corresponding ${\rm P_{LS}}(\lambda^*)$, where the optimal solution $\lambda^*$ to $\varphi(\lambda) = \varrho$ is obtained by SMOP.}}
    \begin{tabular}{ccccc}
    \hline
          & idx   & c     & nnz(x) & $c_{LS}$ \bigstrut\\
    \hline
    \multirow{6}[2]{*}{Test I} & 4     & 0.1   & 6     & 4.4-3 \bigstrut[t]\\
          & 5     & 0.1   & 50    & 2.4-2 \\
          & 6     & 0.15  & 138   & 1.3-2 \\
          & 7     & 0.002 & 28    & 2.4-5 \\
          & 8     & 0.15  & 66    & 8.4-3 \\
    \hline
    \multirow{8}[2]{*}{Test II} & 1     & 0.105 & 95    & 7.5-7 \bigstrut[t]\\
          & 3     & 0.08  & 403   & 4.3-7 \\
          & 4     & 0.08  & 731   & 2.2-4 \\
          & 5     & 0.02  & 120   & 9.1-4 \\
          & 6     & 0.05  & 372   & 6.3-4 \\
          & 7     & 0.001 & 186   & 1.3-6 \\
          & 8     & 0.08  & 260   & 4.9-4 \\
    \hline
    \end{tabular}%
  \label{tab:groupcvalue}%
\end{table}%

\begin{table}[htbp]
  \centering
  \tiny
  \caption{{The performance of SMOP (A1), SSNAL-LSM (A2), SPGL1 (A3) and ADMM (A4), in solving the group lasso penalized problems with least-squares constraints \cref{eq: main-prob} with $\varrho = c\norm{b}$.  The stopping tolerance is set to $10^{-6}$ and the underline is used to mark cases where the algorithm fails to reach the given tolerance. For simplicity, we omit the ``e" in the scientific notation.}}
    \begin{tabular}{|c|c|c|c|}
    \hline
    \multirow{2}[4]{*}{idx} & time (s) & $\eta$ & outermost iter \bigstrut\\
\cline{2-4}          & A1 $|$ A2 $|$ A3  $|$ A4 & A1 $|$ A2 $|$ A3  $|$ A4 & A1 $|$ A2 $|$ A3  $|$ A4 \bigstrut\\
    \hline
    \multicolumn{4}{|c|}{Test I} \bigstrut\\
    \hline
    4     & 3.75+0 $|$ 1.16+2 $|$ 8.49+2 $|$ \underline{3.60+3} & 1.3$-$7 $|$ 3.1-7 $|$ 6.37-7 $|$ \underline{7.2-5} & 11 $|$ 21 $|$ 3024 $|$ \underline{22125} \bigstrut[t]\\
    5     & 8.14$-$1 $|$ 2.74+2 $|$ 2.96+1 $|$ 9.16+2 & 1.4$-$9 $|$ 3.5-7 $|$ 6.05-7 $|$ 1.0-6 & 11 $|$ 21 $|$ 1319 $|$ 38530 \\
    6     & 5.19+0 $|$ 1.46+3 $|$ 1.70+2 $|$ 3.02+3 & 3.2-10 $|$ 4.5-7 $|$ 5.98-7 $|$ 9.8-7 & 10 $|$ 22 $|$ 1086 $|$ 15768 \\
    7     & 5.98$-$1 $|$ 8.80+0 $|$ 3.02+1 $|$ 2.59+1 & 3.7$-$8 $|$ 5.0-7 $|$ 2.07-7 $|$ 1.0-6 & 14 $|$ 19 $|$ 2102 $|$ 1627 \\
    8     & 6.88$-$1 $|$ 1.41+2 $|$ 8.30+0 $|$ 1.19+2 & 1.8$-$8 $|$ 2.6-7 $|$ 2.46-7 $|$ 9.6-7 & 9 $|$ 22 $|$ 334 $|$ 6211 \\
    \hline
    \multicolumn{4}{|c|}{Test II} \bigstrut\\
    \hline
    1     & 3.29+0 $|$ 4.33+1 $|$ 3.18+3 $|$ 1.12+3 & 2.7-7 $|$ 2.7-7 $|$ 9.8-7 $|$ 1.0-6 & 24 $|$ 29 $|$ 55596 $|$ 5826 \bigstrut[t]\\
    3     & 3.83+0 $|$ 3.00+1 $|$ \underline{2.06+3} $|$ 2.57+2 & 1.3-7 $|$ 3.4-7 $|$ \underline{3.8-6} $|$ 1.0-6 & 22 $|$ 36 $|$ \underline{100000} $|$ 13031 \\
    4     & 2.97+1 $|$ 2.42+3 $|$ 1.19+3 $|$ 5.86+2 & 5.2-7 $|$ 9.6-9 $|$ 8.6-7 $|$ 7.4-7 & 13 $|$ 27 $|$ 4241 $|$ 3401 \\
    5     & 1.70+0 $|$ 1.29+2 $|$ 3.30+2 $|$ 9.27+1 & 8.0-7 $|$ 1.7-8 $|$ 8.9-7 $|$ 6.5-7 & 9 $|$ 20 $|$ 18001 $|$ 3959 \\
    6     & 2.51+1 $|$ 1.39+3 $|$ \underline{3.60+3} $|$ 3.60+3 & 1.3-8 $|$ 1.4-7 $|$ \underline{5.8-5} $|$ 2.6-7 & 11 $|$ 22 $|$ \underline{20646} $|$ 19075 \\
    7     & 1.22+0 $|$ 1.88+1 $|$ 5.99+2 $|$ 2.69+1 & 5.3-8 $|$ 2.1-8 $|$ 6.9-7 $|$ 9.9-7 & 15 $|$ 23 $|$ 41578 $|$ 1685 \\
    8     & 5.75+0 $|$ 1.47+2 $|$ 1.14+2 $|$ 1.94+2 & 2.5-7 $|$ 3.7-7 $|$ 4.4-7 $|$ 9.8-7 & 15 $|$ 25 $|$ 4373 $|$ 9974 \\
    \hline
    \end{tabular}%
  \label{tab:groupcomp}%
\end{table}%

\begin{figure}
	\centering
		\includegraphics[width=0.8\textwidth]{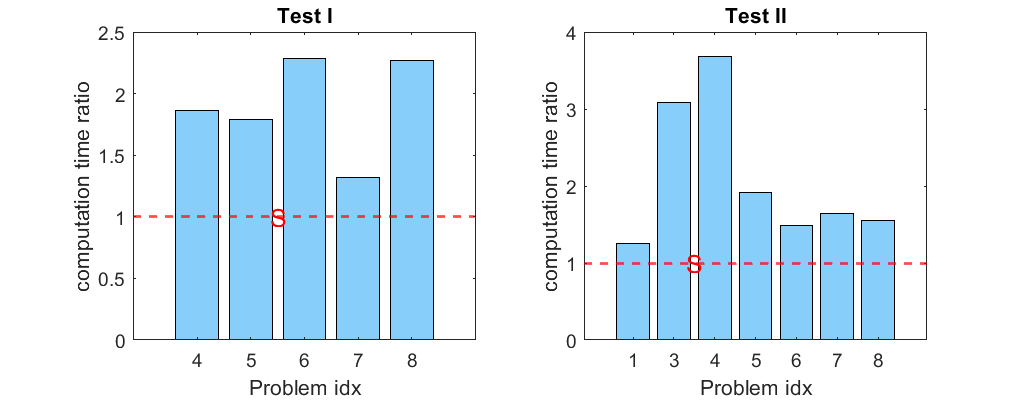}
 	\caption{{The ratio of the computation time between BMOP to the computation time of SMOP in solving \cref{eq: main-prob} with the group lasso regularization}.}
	\label{fig:ratioBSgroup}
 \vspace{-0.6cm}
\end{figure}

In this subsection, we will present the numerical experiments conducted to solve a group lasso penalized problems with least-squares constraints. The purpose of this demonstration is to illustrate the potential and high efficiency of our proposed secant method in solving the equation $\varphi (\lambda) = \varrho$ for the non-polyhedral function penalized problems with least squares constraints. We will compare our algorithm, SMOP, with other state-of-the-art algorithms to demonstrate its high efficiency and robustness.

We consider the following penalty function $p(\cdot)$ in this subsection:
\begin{equation}
\label{eq: group-main}
p(x) = \sum_{t=1}^{l} \sqrt{x_{2t-1}^2 + x_{2t}^2}, \quad x \in \mathbb{R}^{2l}.
\end{equation}
For the purpose of demonstration, we will keep using the UCI dataset that was utilized in the previous two subsections. However, it is necessary to ensure that the value of n is even. Next, we group the $i$-th and $(i+1)$-th elements together for all $i = 1, 3, \cdots, n-1$. The values of $c$ utilized to obtain $\varrho = c \norm{b}$ are presented in Table \ref{tab:groupcvalue}. In SMOP, the sub-problems are solved by SSNAL \cite{zhang2020efficient}. The maximum iteration number for both SMOP and SSNAL-LSM is set to 200, while for SPGL1 and ADMM, their maximum iteration number is set to 100,000. As for the maximum running time, it remains set at 1 hour. Next, we will compare SMOP with the state-of-the-art algorithms SSNAL-LSM, SPGL1, and ADMM. The results of the tests are presented in Table \ref{tab:groupcomp}. From the table, it is evident that SMOP outperforms SSNAL-LSM, SPGL1, and ADMM with speed-ups of up to 300, 900, and 1,100, respectively. In addition, Figure \ref{fig:ratioBSgroup} illustrates the ratio of computation time between BMOP and SMOP. This figure clearly shows that using the secant method can greatly enhance overall efficiency, resulting in a speed improvement of approximately 1.5-4 times, even when dealing with the non-polyhedral penalty function \cref{eq: group-main}.

\section{Conclusion}
\label{sec: conclusion}
In this paper, we have designed an efficient sieving based secant method for solving \cref{eq: main-prob}. When $p(\cdot)$ is a polyhedral gauge function, we have proven that for any $\bar{\lambda} \in (0, \lambda_{\infty})$,  all  $v \in \partial \varphi(\bar{\lambda})$ are positive. Consequently, when $p(\cdot)$ is a polyhedral gauge function, the secant method can solve \cref{eq: root-finding-levelset-lst} with at least a 3-step Q-quadratic convergence rate.  We have demonstrated the high efficiency of our method for solving \cref{eq: main-prob} by two representative instances, specifically, the $\ell_1$ and the sorted $\ell_1$ penalized constrained problems. It is worth mentioning that calculating $\partial_{\rm HS}\varphi(\cdot)$ or $\partial\varphi(\cdot)$ is not an easy task for the sorted $\ell_1$ penalized constrained problems, to the best of our knowledge. Moreover, our numerical results on the $\ell_1$ penalized constrained problems, in which the  $\partial_{\rm HS}\varphi(\cdot)$ is computable as shown in Proposition \ref{eq:HS-Jacobian-ell1},  have verified  that the efficiency of SMOP is not compromised compared to the performance of the HS-Jacobian based semismooth Newton method. This motivates  us to use the secant method instead of the semismooth Newton method for solving \cref{eq: root-finding-levelset-lst} regardless of the availability of the generalized Jacobians. For future research, we will investigate the properties of $ \varphi(\cdot)$ for non-polyhedral functions $p(\cdot)$, {particularly when the nondegenerate condition does not hold.
}

\section*{Acknowledgments}
{The authors would like to thank the referees and the associate editor for their valuable suggestions  to improve  the quality of this paper. Thanks also go to Mr. Jiaming Ma at The Hong Kong Polytechnic University for his helpful discussions on  the proof of \cref{prop: strict-monotonicity-tame}{ (ii)}.}

\bibliographystyle{siamplain}
\bibliography{references_SMOP}
\end{document}


\maketitle

\section{Supplementary proof of Proposition 3.3(ii)}

Assume that 
\begin{equation}
   +\infty > \lambda_{\infty} := \Upsilon(A^T b\,|\,\partial p(0)) > 0.
\end{equation}

The condition $\lambda_{\infty} < + \infty$ implies that
\[
p^{\circ}(A^Tb) \leq \lambda_{\infty} < + \infty
\]
and
\[
A^Tb \in {\rm dom}(p^{\circ}).
\] 

{\color{red}{Assume that 
\[
{\rm Range}(A^T) \bigcap {\rm ri}({\rm dom}p^{\circ}) \neq \emptyset.
\]
}}

Assume that there exists $0< \lambda_1 < \lambda_2 \leq \lambda_{\infty}$ such that $\varphi(\lambda_1) = \varphi(\lambda_2)$, then we must have 
\[
p^{\circ}(A^Ty(\lambda_2)) = p^{\circ}(A^Ty(\lambda_1)) \leq \lambda_1 < \lambda_2 < +\infty.
\]

Note that $(y(\lambda_2), z(\lambda_2)) = (y(\lambda_2), A^Ty(\lambda_2))$ is the solution to 
\begin{equation}
\label{eq: dual}
\min_{y \in \mathbb{R}^m, z \in \mathbb{R}^n} ~ \left\{ \frac{1}{2}\|y\|^2 - \langle b, y \rangle ~|~ A^Ty - z = 0, ~ p^{\circ}(z) \leq \lambda_2 \right\},
\end{equation}

and Lagrangian multiplier(s) to the problem \cref{eq: dual} exist \cite[Theorem 28.2]{rockafellar1970convex}. 

Define
\[
l(y, z; \alpha, x) := \left\{
\begin{array}{ll}
\frac{1}{2}\|y\|^2 - \langle b, y \rangle + \alpha(p^{\circ}(z) - \lambda_2) + \langle x, A^Ty - z\rangle    &  {\rm if } ~ \alpha \geq 0, z \in {\rm dom} p^{\circ},\\
-\infty   & {\rm if} ~ \alpha < 0, z \in {\rm dom} p^{\circ},\\
+\infty & {\rm if} ~ z\not\in {\rm dom} p^{\circ}.
\end{array}
\right.
\]

Then, there exists $(\bar{\alpha}, \bar{x}) \in \mathbb{R}_+ \times \mathbb{R}^n$ such that the Karush-Kuhn-Tucker (KKT) system is satisfied:
\[
\left\{
\begin{array}{l}
\bar{\alpha} \geq 0, ~ p^{\circ}(z(\lambda_2)) < \lambda_2, ~ \bar{\alpha}(p^{\circ}(z(\lambda_2)) - \lambda_2) = 0,\\
A^Ty(\lambda_2) - z(\lambda_2) = 0,\\
y(\lambda_2) - b + A\bar{x} = 0,\\
0 = -\bar{x}.
\end{array}
\right.
\]

The term $\bar{\alpha}\partial p^{\circ}(z(\lambda_2))$ is omitted in the last equation since $\bar{\alpha} = 0$ \cite[Theorem 28.3]{rockafellar1970convex}.

Therefore, $0 \in \Omega(\lambda_2)$, $y(\lambda_2) = b$, and $z(\lambda_2) = A^Tb$. This is a contradiction.

\newpage

\bibliographystyle{siamplain}
\bibliography{references_SMOP}

%% file: ex_shared.tex
\usepackage{lipsum}
\usepackage{amsfonts}
\usepackage{graphicx}
\usepackage{epstopdf}
\usepackage{algorithmic}
\ifpdf
  \DeclareGraphicsExtensions{.eps,.pdf,.png,.jpg}
\else
  \DeclareGraphicsExtensions{.eps}
\fi

\newsiamremark{remark}{Remark}
\newsiamremark{hypothesis}{Hypothesis}
\crefname{hypothesis}{Hypothesis}{Hypotheses}
\newsiamthm{claim}{Claim}

\headers{An efficient sieving based secant method}{Qian Li, Defeng Sun, and Yancheng Yuan}

\title{An efficient sieving based secant method for sparse optimization problems with least-squares constraints\thanks{
\funding{The research  of Qian Li was supported in part  by Huawei Collaborative Grants ``Large scale linear programming solver" and  ``Solving large scale linear programming models for production planning". The research  of Defeng Sun   was supported   by  grants from the Research Grants Council of the Hong Kong Special Administrative Region, China (GRF Project No. 15304721 and  RGC Senior Research Fellow Scheme No. SRFS2223-5S02).  The research of Yancheng Yuan was supported by the Hong Kong Polytechnic University under grant P0038284.
}}}

\author{Qian Li\thanks{Department of Applied Mathematics, The Hong Kong Polytechnic University, Hung Hom, Hong Kong
  (\email{qianxa.li@connect.polyu.hk}).}
  \and Defeng Sun\thanks{Department of Applied Mathematics, The Hong Kong Polytechnic University, Hung Hom, Hong Kong (\email{defeng.sun@polyu.edu.hk}).}
  \and Yancheng Yuan\thanks{Department of Applied Mathematics, The Hong Kong Polytechnic University, Hung Hom, Hong Kong
  		(\email{yancheng.yuan@polyu.edu.hk}).}
}

\usepackage{amsopn}

\makeatletter
\newcommand*{\addFileDependency}[1]{
  \typeout{(#1)}
  \@addtofilelist{#1}
  \IfFileExists{#1}{}{\typeout{No file #1.}}
}
\makeatother
